\documentclass[11pt]{amsart}

\usepackage{amsmath,amssymb,amsthm,mathtools}
\usepackage{enumitem,booktabs,microtype}
\usepackage[hidelinks]{hyperref}

\newtheorem{theorem}{Theorem}[section]
\newtheorem{proposition}[theorem]{Proposition}
\newtheorem{lemma}[theorem]{Lemma}
\newtheorem{corollary}[theorem]{Corollary}
\theoremstyle{definition}
\newtheorem{definition}[theorem]{Definition}
\theoremstyle{remark}
\newtheorem{remark}[theorem]{Remark}
\newtheorem{example}[theorem]{Example}
\newtheorem{problem}[theorem]{Problem}

\newcommand{\R}{\mathbb R}
\newcommand{\cH}{\mathcal H}
\newcommand{\supp}{\operatorname{supp}}

\newcommand{\dd}{\,\mathrm d}
\newcommand{\asympC}{\mathrel{\asymp}}

\title[Completion Fibres of Newton Conformal Metrics]{Completion Fibres of Infinite Positive Newton Conformal Metrics Near Corners}
\author{Muhamad Fahmi bin Zanal Abidin}
\address{Independent Researcher, Kuala Lumpur, Malaysia}
\email{fahmiemail123@gmail.com}
\thanks{ORCID: 0009-0009-1024-9858}
\date{}

\begin{document}
\begin{abstract}
We study the local metric completion of positive singular conformal metrics near a codimension-two corner when the conformal factor is generated by a compact positive measure on exponent space. The Newton support function determines noncritical weighted-ray accessibility, but at the critical power order support geometry alone is insufficient. We identify the missing invariant as a Laplace transform measuring how exponent mass approaches the critical total-degree face. If the maximal total degree exceeds \(2\), the local completion fibre over the corner is empty. When all total degrees are at most \(2\), the fibre is empty or a singleton according as the square root of the critical Laplace profile is nonintegrable or integrable at infinity. We give support-identical measures with opposite critical completion behaviour, sharpness examples, monotone truncation phenomena, intrinsic Lipschitz completion coordinates, and a bi-Lipschitz transfer principle for Newton-comparable densities.
\end{abstract}
\maketitle

\section{Introduction}

Newton polyhedra and related support constructions have long been used to
organize competing monomial orders in singularity theory and asymptotic
analysis.  Toric resolution associated with Newton polyhedra is classical
\cite{Khovanskii1984}.  More recent work develops local monomialization,
combinatorial monomialization, and stratified resolution for generalized
analytic functions whose monomials may have nonnegative real exponents
\cite{MartinRolinSanz2013,Palma2022,MolinaPalmaSanz2024,Palma2024}.
Accordingly, neither generalized real-exponent series nor non-rational
Newton polyhedra are new objects.

The purpose of the present paper is different.  We study the intrinsic
length geometry of positive singular conformal metrics generated by an
infinite distribution of inverse monomials near a codimension-two corner.
Related precursor work treats product and sum multi-weighted conformal
metrics on manifolds with corners and their metric completions
\cite{ZanalAbidinPaper1}.  The finite positive Newton setting, including the
projective support-function criterion and induced boundary geometry, is
developed in \cite{ZanalAbidinPaper3}.  For self-containment, the finite
comparison needed here is also proved directly.  In that setting a weighted
direction \(w=(p,q)\) is accessible precisely when the finite support
function satisfies
\[
\cH(w)<2\min\{p,q\}.
\]
The present paper passes from finite support to a compact positive exponent
measure.  For infinite support, the support function continues to determine
the leading logarithmic order, but it no longer determines the critical
metric behaviour.  This
is a fixed-measure instance of the classical Laplace principle
\cite{Varadhan1966}; related log-Laplace transforms are also standard in
convex geometry \cite{KlartagMilman2012}.  It is therefore used here as
background rather than claimed as a new asymptotic theorem.  What fails in
the infinite-support setting is the stronger two-sided monomial comparison
at critical directions.

The central new feature is an exposed-face tail.  Along
\[
x=r^p,\qquad y=r^q,\qquad s=\log(1/r),
\]
the conformal factor admits an exact factorization
\[
F_\mu(r^p,r^q)
=
r^{-\cH_\mu(p,q)}\Lambda_{\mu,(p,q)}(s),
\]
where \(\Lambda_{\mu,(p,q)}\) is a Laplace transform of the mass-depth
distribution relative to the exposed Newton face.  If the exposed face has
positive mass, the finite-support power law is recovered.  If the exposed
face has zero mass, the Laplace factor tends to zero and may alter the
metric behavior at the critical exponent.

The measure formulation is adopted as a unified ambient language rather
than as a claim of necessity.  Countably atomic measures already exhibit
the new critical-tail phenomenon.  General finite Radon measures are
included because they provide a coordinate-free closure of the atomic
class, separate support geometry from coefficient mass, and allow discrete,
mixed, and diffuse exponent distributions to be treated by the same
notation.  The main metric arguments will first be proved in the atomic
setting whenever this yields a clearer statement, and then extended to the
measure setting by direct integral estimates.

\paragraph{Relation to conformal-density boundary theory.}
General conformal-density theory already studies accessibility and the metric
boundary of a conformally deformed domain.  Bonk and Koskela
\cite{BonkKoskela2002} obtain lower bounds for the Hausdorff dimension of the
metric boundary of conformal deformations of the Euclidean ball and derive
consequences for quasiconformal images.  Their principal results concern the
size of the completed boundary under Harnack, volume-growth, and
Gehring--Hayman-type hypotheses; they do not give a pointwise accessibility
criterion in terms of an exponent measure, nor do they classify the completion
fibre over a fixed Euclidean boundary point.

Nieminen \cite{Nieminen2009} studies broadly accessible boundary points for
conformal metrics on the Euclidean ball.  His theorems extend a result of
Gerasch by estimating, under Harnack and isodiametric-profile assumptions, the
size of the exceptional set of boundary points that fail broad accessibility.
This is close in geometric motivation to the present work, but logically
different: the conclusions are Hausdorff-size statements for sets of boundary
points, whereas the present paper gives an exact local criterion for one
corner point and identifies its local metric-completion fibre.  Neither paper
contains the Newton-support reduction, the critical exposed-face Laplace tail,
or the empty-or-singleton fibre alternative proved below.

Kl\'en and Suomala \cite{KlenSuomala2013} study Hausdorff and packing
dimensions of metric boundaries induced by continuous densities on Euclidean
domains.  Their framework explicitly compares the original and density-metric
completions and, in radial and power-law regimes, derives boundary-dimension
formulas from asymptotic density exponents.  This is the closest adjacent
result in terms of converting an integral density profile into boundary
geometry.  Their principal object is nevertheless the dimension of the whole
metric boundary, whereas the present work treats a prescribed corner point,
allows a compact distribution of competing monomial exponents, and classifies
the corresponding local completion fibre.

\paragraph{Relation to resolution and singular-metric geometry.}
Generalized boundary blow-up on manifolds with corners, as developed by
Kottke and Melrose \cite{KottkeMelrose2015}, resolves competing homogeneities
through refinements of monoidal boundary data.  Grandjean
\cite{Grandjean2019} proves monomialization results for real-analytic singular
metrics on surfaces, while Grandjean and Grieser \cite{GrandjeanGrieser2018}
study geodesics and the exponential map for cuspidal metrics.  Conformal
Grushin spaces provide another nearby model in which a prescribed power-law
density changes the metric geometry of a singular set \cite{Romney2016}.
Soares J\'unior, Costa, and Saia \cite{SoaresCostaSaia2026} give a recent
example in singularity theory where Newton filtration is used to define a
control function and a singular metric for bi-Lipschitz and differentiable
\(\mathcal R\)-sufficiency of jets.  This confirms that combining Newton
data with singular metric structures has modern precedent, but their problem
is jet sufficiency rather than intrinsic path distance, accessibility, or
metric completion.  These works supply resolution-theoretic, geodesic, and
model-metric context, but they do not formulate the conformal factor through
a compact exponent measure or derive the total-degree support-and-tail
criterion used here to classify a local completion fibre.

\paragraph{Scope and novelty.}
The paper does not claim a new theory of generalized power series, Newton
polyhedra, resolution, or conformal boundaries.  Its contribution is a direction-uniform local completion theorem for
densities generated by a compact positive exponent measure near a prescribed
codimension-two corner.  Support geometry controls the noncritical regimes,
while the distribution of mass near the critical total-degree face enters
through a sharp Laplace-tail integral.

\subsection*{Main results}
Set
\[
A_\mu=\max_{\alpha\in\operatorname{supp}\mu}(\alpha_1+\alpha_2).
\]
Theorem~\ref{thm:complete-local-completion-classification} proves the
following dichotomy for the local completion fibre over a prescribed corner
point:
\begin{enumerate}[label=(\roman*)]
\item if \(A_\mu>2\), the fibre is empty;
\item if \(A_\mu\le2\), define
\[
\mathcal T_\mu(S)=\int_E e^{-S(2-\alpha_1-\alpha_2)}\,\mathrm d\mu(\alpha).
\]
The fibre is empty when
\(\int^\infty\sqrt{\mathcal T_\mu(S)}\,\mathrm dS=\infty\), and is a
singleton when this integral is finite.
\end{enumerate}
The proof also shows that, in the non-supercritical regime, arbitrary
finite-length accessibility, diagonal accessibility, tail integrability, and
singleton-fibre completion are equivalent.  Theorem~\ref{thm:support-insufficiency}
shows that the closed Newton support and its support function do not determine
the critical outcome: identical support geometry can carry different mass
tails and hence opposite completion behaviour.

\paragraph{Organization.}
Section~\ref{sec:model} introduces the admissible class.  Section~\ref{sec:directional}
defines the support function, exposed faces, and directional Laplace profiles.
Sections~\ref{sec:strict-gap} and~\ref{sec:critical} establish the weighted-ray
classification.  Section~\ref{sec:global} separates radial from unrestricted
accessibility.  Section~\ref{sec:metric-completion} proves the local
metric-completion classification.  Section~\ref{sec:truncations} treats finite
truncations, Section~\ref{sec:classification} identifies the directional data,
the final sections discuss classification data, sharpness, scope, and Newton-comparable densities.

\section{Admissible infinite Newton interactions}
\label{sec:model}

Let \(X\) be a smooth manifold with corners and let
\[
C=H_1\cap H_2
\]
be a codimension-two corner stratum.  In an adapted coordinate
neighborhood, write
\[
U=[0,\varepsilon)^2\times V,\qquad C=\{x=y=0\},
\]
where \(z\in V\subset\R^k\) denotes the tangential variables.  Let \(g_0\)
be a smooth Riemannian metric satisfying
\begin{equation}\label{eq:background-ellipticity}
c_0 g_{\mathrm E}\le g_0\le C_0 g_{\mathrm E}
\end{equation}
on \(U\), for constants \(0<c_0\le C_0<\infty\).

\begin{definition}[Positive compactly supported Newton interaction]
\label{def:measure-interaction}
Let \(E\subset\R_{\ge0}^2\) be compact and let \(\mu\) be a finite, nonzero, positive
Radon measure with \(\supp\mu=E\).  Let
\[
u:E\times U^\circ\longrightarrow(0,\infty)
\]
be jointly measurable.  Assume in addition that the integral in
\eqref{eq:measure-model} defines a continuous positive function on
\(U^\circ\) (for example, this follows if \(u\) is continuous and locally
uniformly dominated in the interior), and assume that
\begin{equation}\label{eq:coefficient-bounds}
0<c_u\le u(\alpha;x,y,z)\le C_u<\infty
\end{equation}
for \(\mu\)-almost every \(\alpha\in E\) and every
\((x,y,z)\in U^\circ\).  The associated positive Newton interaction is
\begin{equation}\label{eq:measure-model}
F_\mu(x,y,z)
=
\int_E
u(\alpha;x,y,z)x^{-\alpha_1}y^{-\alpha_2}\,\dd\mu(\alpha),
\end{equation}
and the associated singular conformal metric is
\begin{equation}\label{eq:metric-model}
g_\mu=F_\mu g_0.
\end{equation}
\end{definition}

The frozen-coefficient model is
\begin{equation}\label{eq:frozen-model}
F_\mu^0(x,y)
=
\int_E x^{-\alpha_1}y^{-\alpha_2}\,\dd\mu(\alpha).
\end{equation}
By \eqref{eq:coefficient-bounds},
\begin{equation}\label{eq:frozen-comparison}
c_u F_\mu^0\le F_\mu\le C_u F_\mu^0.
\end{equation}
Consequently, all length estimates that are stable under uniform
bi-Lipschitz comparison may first be proved for \(F_\mu^0\).

\begin{proposition}[Interior finiteness]
\label{prop:interior-finiteness}
For every \((x,y)\in(0,\varepsilon)^2\), the quantity
\(F_\mu^0(x,y)\) is finite and strictly positive.  Hence \(g_\mu\) is a
well-defined Riemannian metric on \(U^\circ\).
\end{proposition}

\begin{proof}
Compactness of \(E\) gives finite numbers
\[
A_0=\max_{\alpha\in E}\alpha_1,\qquad
B_0=\max_{\alpha\in E}\alpha_2.
\]
For \(x,y>0\),
\[
x^{-\alpha_1}y^{-\alpha_2}
\le
\max\{1,x^{-A_0}\}\max\{1,y^{-B_0}\}
\]
for every \(\alpha\in E\).  Since \(\mu(E)<\infty\),
\[
F_\mu^0(x,y)
\le
\mu(E)\max\{1,x^{-A_0}\}\max\{1,y^{-B_0}\}<\infty.
\]
Strict positivity follows from positivity of the integrand and the
assumption that \(\mu\) is nonzero.
\end{proof}

\subsection{Atomic and diffuse formulations}

\begin{definition}[Countably atomic Newton interaction]
\label{def:atomic-interaction}
A positive Newton interaction is countably atomic if
\[
\mu=\sum_{j=1}^{\infty}c_j\delta_{\alpha_j},
\qquad
c_j>0,\qquad
\sum_{j=1}^{\infty}c_j<\infty,
\]
where \(\alpha_j=(a_j,b_j)\) belongs to a fixed compact subset of
\(\R_{\ge0}^2\).
\end{definition}

In this case,
\begin{equation}\label{eq:atomic-series}
F_\mu^0(x,y)
=
\sum_{j=1}^{\infty}
c_jx^{-a_j}y^{-b_j}.
\end{equation}
Finite positive Newton sums are recovered when only finitely many atoms are
present.

\begin{proposition}[Atomic realization]
\label{prop:atomic-realization}
Every countable positive series of the form \eqref{eq:atomic-series}, with
summable positive coefficients and compact exponent closure, is exactly a
positive compactly supported Newton interaction.  Conversely, every
countably atomic interaction has the form \eqref{eq:atomic-series}.
\end{proposition}

\begin{proof}
Given the series, define
\[
\mu=\sum_{j=1}^{\infty}c_j\delta_{\alpha_j}.
\]
Summability gives \(\mu(E)<\infty\), and direct evaluation of the integral
in \eqref{eq:frozen-model} gives \eqref{eq:atomic-series}.  The converse is
the same calculation in reverse.
\end{proof}

\begin{remark}[Why the measure model is retained]
\label{rem:why-measures}
The countably atomic class is sufficient to generate nonattained
coefficient tails and critical logarithmic corrections.  Thus diffuse
measures are not required merely to produce the main new phenomenon.
Nevertheless, the measure formulation is retained because:
\begin{enumerate}[label=(\roman*)]
\item it separates exponent support from coefficient mass;
\item it treats finite, countably atomic, mixed, and diffuse distributions
in one class;
\item exposed-face tails are pushforward measures; their Laplace profiles
are therefore canonical;
\item weak limits of atomic models remain within the same ambient class.
\end{enumerate}
No approximation or convergence theorem is asserted here; those require
separate metric hypotheses and will be proved later.
\end{remark}

\section{Directional support and exposed-face tails}
\label{sec:directional}

For \(w=(p,q)\in\R_{>0}^2\), define the support function
\begin{equation}\label{eq:support-function}
\cH_\mu(w)
=
\max_{\alpha\in E}\langle w,\alpha\rangle
=
\max_{\alpha\in E}(p\alpha_1+q\alpha_2).
\end{equation}
Compactness of \(E\) guarantees attainment.  The function \(\cH_\mu\) is
positively homogeneous, convex, and continuous.  These are standard facts
from convex analysis and are not claimed as new results.

\begin{definition}[Exposed Newton face]
\label{def:exposed-face}
The exposed face in direction \(w\) is
\[
E_w
=
\{\alpha\in E:
\langle w,\alpha\rangle=\cH_\mu(w)\}.
\]
\end{definition}

Define the depth map
\begin{equation}\label{eq:depth-map}
\Delta_w(\alpha)
=
\cH_\mu(w)-\langle w,\alpha\rangle\ge0
\end{equation}
and the associated pushforward measure
\begin{equation}\label{eq:depth-measure}
\nu_w=(\Delta_w)_\#\mu.
\end{equation}
Its distribution function near the origin is the directional tail
\begin{equation}\label{eq:tail-function}
T_{\mu,w}(\delta)
=
\nu_w([0,\delta])
=
\mu\{\alpha\in E:\Delta_w(\alpha)\le\delta\}.
\end{equation}

\begin{definition}[Directional Laplace profile]
\label{def:laplace-profile}
For \(s\ge0\), define
\begin{equation}\label{eq:laplace-profile}
\Lambda_{\mu,w}(s)
=
\int_E e^{-s\Delta_w(\alpha)}\,\dd\mu(\alpha)
=
\int_{[0,\infty)}e^{-st}\,\dd\nu_w(t).
\end{equation}
\end{definition}

\begin{proposition}[Exact directional factorization]
\label{prop:exact-factorization}
Let \(w=(p,q)\in\R_{>0}^2\), let \(0<r<1\), and put
\(s=\log(1/r)\).  Then
\begin{equation}\label{eq:exact-factorization}
F_\mu^0(r^p,r^q)
=
r^{-\cH_\mu(w)}
\Lambda_{\mu,w}(s).
\end{equation}
Moreover,
\begin{equation}\label{eq:laplace-limit}
\lim_{s\to\infty}\Lambda_{\mu,w}(s)=\mu(E_w).
\end{equation}
\end{proposition}

\begin{proof}
Since
\[
(r^p)^{-\alpha_1}(r^q)^{-\alpha_2}
=
r^{-\langle w,\alpha\rangle}
=
r^{-\cH_\mu(w)}
e^{-s\Delta_w(\alpha)},
\]
integration over \(E\) gives \eqref{eq:exact-factorization}.  For each
\(\alpha\in E\),
\[
e^{-s\Delta_w(\alpha)}
\longrightarrow
\mathbf 1_{E_w}(\alpha)
\]
as \(s\to\infty\), and the integrand is bounded by \(1\).  Dominated
convergence therefore yields \eqref{eq:laplace-limit}.
\end{proof}

\begin{corollary}[Positive exposed-face mass]
\label{cor:positive-face-mass}
If \(\mu(E_w)>0\), then
\[
F_\mu^0(r^p,r^q)\asympC r^{-\cH_\mu(w)}
\qquad (r\downarrow0).
\]
\end{corollary}

\begin{proof}
By Proposition~\ref{prop:exact-factorization},
\(\Lambda_{\mu,w}(s)\to\mu(E_w)>0\).  Hence it is bounded above and below by
positive constants for all sufficiently large \(s\).
\end{proof}

\begin{corollary}[Zero exposed-face mass]
\label{cor:zero-face-mass}
If \(\mu(E_w)=0\), then
\[
F_\mu^0(r^p,r^q)
=
o\!\left(r^{-\cH_\mu(w)}\right)
\qquad (r\downarrow0).
\]
\end{corollary}

\begin{proof}
This follows from
\(\Lambda_{\mu,w}(s)\to0\) in
Proposition~\ref{prop:exact-factorization}.
\end{proof}

\subsection{Atomic and measure formulations}

The preceding results settle the atomic-versus-measure design question.
A countably atomic measure can have an exponent sequence approaching an
exposed face while assigning no mass to the limiting face itself.  Such a
series already satisfies
\[
\Lambda_{\mu,w}(s)\longrightarrow0
\]
and therefore departs from finite-support two-sided monomial asymptotics.
The measure formulation is consequently not needed to create the
phenomenon, but it is the natural ambient class in which the relevant
invariant is expressed.

It is useful to distinguish three nested subclasses:
\begin{enumerate}[label=(\arabic*)]
\item finite atomic measures, which recover finite positive Newton sums;
\item countably atomic measures, which already exhibit the critical-tail phenomenon;
\item arbitrary finite positive Radon measures, which provide the natural unified class.
\end{enumerate}
The proofs below are written at the measure level whenever no additional
argument is required; atomic examples are used to display the mechanisms
explicitly.

\section{Logarithmic order and strict-gap radial geometry}
\label{sec:strict-gap}

\subsection{Relation to classical asymptotics}

The limit of a scaled logarithmic exponential integral to the supremum of
its phase is a classical Laplace-principle phenomenon; in probabilistic
language it is closely related to Varadhan's asymptotic lemma
\cite{Varadhan1966}.  The same integral is the log-Laplace transform of the
exponent measure, an object widely used in convex geometry
\cite{KlartagMilman2012}.  In finite atomic notation it is also the
zero-temperature limit of a log-sum-exp function, while in Newton notation
it is the passage from a positive sum to the maximal weighted exponent.
Thus Proposition~\ref{prop:log-order} is included for completeness and to
fix hypotheses, but is not presented as an original theorem.

Adjacent generalized-series papers study monomialization and stratified
resolution rather than intrinsic length geometry
\cite{MartinRolinSanz2013,Palma2022,MolinaPalmaSanz2024,Palma2024}.  The
contribution below therefore begins not with the logarithmic-order limit, but
with the use of the exposed-face Laplace profile as an exact finite-distance
and completion invariant for the Newton metric class.

\begin{proposition}[Directional logarithmic order]
\label{prop:log-order}
For every \(w=(p,q)\in\R_{>0}^2\),
\begin{equation}\label{eq:log-order}
\lim_{r\downarrow0}
\frac{\log F_\mu^0(r^p,r^q)}{\log(1/r)}
=
\cH_\mu(w).
\end{equation}
The same limit holds with \(F_\mu\) in place of \(F_\mu^0\), uniformly in
\(z\), under the coefficient bounds \eqref{eq:coefficient-bounds}.
\end{proposition}

\begin{proof}
Put \(s=\log(1/r)\).  The upper bound follows immediately from
\[
F_\mu^0(r^p,r^q)
\le \mu(E)e^{s\cH_\mu(w)}.
\]
For the lower bound, fix \(\eta>0\).  Since \(E=\supp\mu\) and the maximum
of \(\alpha\mapsto\langle w,\alpha\rangle\) is attained on \(E\), the set
\[
A_\eta=\{\alpha\in E:
\langle w,\alpha\rangle>\cH_\mu(w)-\eta\}
\]
has positive \(\mu\)-measure.  Hence
\[
F_\mu^0(r^p,r^q)
\ge
\mu(A_\eta)e^{s(\cH_\mu(w)-\eta)}.
\]
After taking logarithms and dividing by \(s\), the upper and lower bounds
give
\[
\begin{aligned}
\cH_\mu(w)-\eta
&\le
\liminf_{s\to\infty}\frac1s\log F_\mu^0(e^{-ps},e^{-qs}),\\
\limsup_{s\to\infty}\frac1s\log F_\mu^0(e^{-ps},e^{-qs})
&\le \cH_\mu(w).
\end{aligned}
\]
Letting \(\eta\downarrow0\) proves \eqref{eq:log-order}.  The comparison
\eqref{eq:frozen-comparison} changes the logarithm by a bounded additive
quantity, so the same limit holds for \(F_\mu\), uniformly in \(z\).
\end{proof}

For \(w=(p,q)\), set
\begin{equation}\label{eq:m-w}
m(w)=\min\{p,q\}
\end{equation}
and consider the weighted radial curve
\begin{equation}\label{eq:weighted-curve}
\gamma_w(r)=(r^p,r^q,z_0),\qquad 0<r<r_0.
\end{equation}
Uniform ellipticity of \(g_0\) gives
\begin{equation}\label{eq:curve-speed}
|\dot\gamma_w(r)|_{g_0}\asympC r^{m(w)-1}.
\end{equation}

\begin{theorem}[Strict-gap weighted radial classification]
\label{thm:strict-gap-radial}
Let \(w=(p,q)\in\R_{>0}^2\).  Then:
\begin{enumerate}[label=(\roman*)]
\item If \(\cH_\mu(w)<2m(w)\), the curve \(\gamma_w\) has finite
\(g_\mu\)-length as \(r\downarrow0\).
\item If \(\cH_\mu(w)>2m(w)\), the curve \(\gamma_w\) has infinite
\(g_\mu\)-length as \(r\downarrow0\).
\end{enumerate}
No assertion is made here when \(\cH_\mu(w)=2m(w)\).
\end{theorem}

\begin{proof}
By \eqref{eq:frozen-comparison}, \eqref{eq:background-ellipticity}, and
\eqref{eq:curve-speed}, the radial length is comparable to
\begin{equation}\label{eq:radial-length-model}
\int_0^{r_0}
\bigl(F_\mu^0(r^p,r^q)\bigr)^{1/2}
r^{m(w)-1}\,\dd r.
\end{equation}
If \(\cH_\mu(w)<2m(w)\), then
\[
F_\mu^0(r^p,r^q)
\le \mu(E)r^{-\cH_\mu(w)},
\]
so the integrand in \eqref{eq:radial-length-model} is bounded above by a
constant multiple of
\[
r^{m(w)-1-\cH_\mu(w)/2},
\]
whose exponent is greater than \(-1\).  This proves (i).

Suppose now that \(\cH_\mu(w)>2m(w)\).  Choose \(\eta>0\) such that
\(\cH_\mu(w)-\eta>2m(w)\).  The set \(A_\eta\) used in the proof of
Proposition~\ref{prop:log-order} has positive measure, and therefore
\[
F_\mu^0(r^p,r^q)
\ge \mu(A_\eta)r^{-(\cH_\mu(w)-\eta)}.
\]
The integrand in \eqref{eq:radial-length-model} is bounded below by a
positive constant times
\[
r^{m(w)-1-(\cH_\mu(w)-\eta)/2},
\]
whose exponent is strictly less than \(-1\).  The integral diverges, proving
(ii).
\end{proof}

\begin{remark}[What the theorem does and does not classify]
Theorem~\ref{thm:strict-gap-radial} classifies the length of the prescribed
weighted radial approach.  It does not yet prove a global necessity theorem
for all curves approaching the corner.  Such a statement requires a
pathwise lower-bound argument and will be addressed separately.  This
distinction prevents the directional logarithmic estimate from being used
beyond what it proves.
\end{remark}

\begin{remark}[Why equality is genuinely different]
At equality, Proposition~\ref{prop:exact-factorization} reduces
\eqref{eq:radial-length-model}, after \(s=\log(1/r)\), to an integral
controlled by \(\Lambda_{\mu,w}(s)^{1/2}\).  The support function alone no
longer decides convergence.  Section~\ref{sec:critical} compares atomic and diffuse tails,
constructs counterexamples with identical support function but different
critical lengths, and records the role of Abelian and Tauberian results for
Laplace--Stieltjes transforms.
\end{remark}

\section{Critical radial tails}
\label{sec:critical}

\subsection{Relation to transform asymptotics}

The critical calculation below intersects three established subjects, while
the resulting metric criterion is specific to the present Newton setting.  First, Abelian and Tauberian theory relates
the mass of a measure near an endpoint to the asymptotics of its
Laplace--Stieltjes transform; regular-variation refinements are classical
and continue to be developed \cite{deHaan1977}.  Those results can be used
to evaluate particular profiles \(\Lambda_{\mu,w}\), but they do not by
themselves identify a Riemannian finite-distance threshold.  Second,
Grushin-type geometries exhibit singular distance laws, snowflaking, and
quasisymmetric boundary phenomena \cite{Wu2015}; the conformal factors in
that literature are prescribed directly rather than generated by an
infinite Newton exponent distribution.  Third, generalized Newton-series
papers resolve or monomialize generalized analytic singularities
\cite{MartinRolinSanz2013,Palma2022,MolinaPalmaSanz2024,Palma2024}, but do
not appear to connect exposed-face mass tails with metric completion.

The integral appearing in Theorem~\ref{thm:critical-tail} is elementary
once the radial pullback is written down, and its transform asymptotics are
classical in regularly varying subclasses.  The contribution here is to
identify this integral as the critical invariant for infinite positive
Newton metrics and to show, by examples, that the Newton support and support
function alone are insufficient.  This remains
a non-duplication assessment rather than an absolute priority claim.

\begin{theorem}[Critical-tail criterion for weighted radial length]
\label{thm:critical-tail}
Let \(w=(p,q)\in\R_{>0}^2\), let \(m(w)=\min\{p,q\}\), and suppose
\begin{equation}\label{eq:critical-equality}
\cH_\mu(w)=2m(w).
\end{equation}
Then the weighted radial curve \(\gamma_w(r)=(r^p,r^q,z_0)\) has finite
\(g_\mu\)-length near \(r=0\) if and only if
\begin{equation}\label{eq:critical-tail-integral}
\int_{s_0}^{\infty}\Lambda_{\mu,w}(s)^{1/2}\,\dd s<\infty,
\qquad s_0=\log(1/r_0).
\end{equation}
The equivalence is uniform in \(z_0\) on compact tangential coordinate
sets on which the bounds for \(g_0\) and \(u\) are uniform.
\end{theorem}

\begin{proof}
By \eqref{eq:background-ellipticity},
\eqref{eq:coefficient-bounds}, and \eqref{eq:curve-speed},
\begin{align*}
L_{g_\mu}(\gamma_w)
&\asympC
\int_0^{r_0}
\bigl(F_\mu^0(r^p,r^q)\bigr)^{1/2}
 r^{m(w)-1}\,\dd r.
\end{align*}
Proposition~\ref{prop:exact-factorization} and
\eqref{eq:critical-equality} give
\[
\bigl(F_\mu^0(r^p,r^q)\bigr)^{1/2}
=
r^{-m(w)}
\Lambda_{\mu,w}\!\left(\log\frac1r\right)^{1/2}.
\]
Consequently,
\[
L_{g_\mu}(\gamma_w)
\asympC
\int_0^{r_0}
\Lambda_{\mu,w}\!\left(\log\frac1r\right)^{1/2}
\frac{\dd r}{r}.
\]
The substitution \(s=\log(1/r)\) transforms the last integral exactly
into
\[
\int_{s_0}^{\infty}\Lambda_{\mu,w}(s)^{1/2}\,\dd s.
\]
The two-sided metric and coefficient comparisons prove the equivalence.
\end{proof}

\begin{corollary}[Positive mass on the exposed face]
\label{cor:critical-face-atom}
Under the critical equality \eqref{eq:critical-equality}, if
\(\mu(E_w)>0\), then \(\gamma_w\) has infinite \(g_\mu\)-length.
\end{corollary}

\begin{proof}
By Proposition~\ref{prop:exact-factorization},
\(\Lambda_{\mu,w}(s)\to\mu(E_w)>0\).  Hence the integral in
\eqref{eq:critical-tail-integral} diverges.
\end{proof}

\begin{remark}[The equality case in the finite theory]
Every finite positive Newton interaction assigns positive mass to at least
one exponent on each exposed face.  Corollary~\ref{cor:critical-face-atom}
therefore recovers the divergence at equality from the finite-support
setting.  Critical accessibility can occur only after the leading face has
zero mass and is approached by infinitely many exponents or by diffuse
mass.
\end{remark}

\subsection{Regularly varying critical profiles}

Theorem~\ref{thm:critical-tail} is exact and requires no regular-variation
assumption.  Such assumptions become useful only for evaluating the tail
integral.

\begin{corollary}[Power-law Laplace profiles]
\label{cor:power-profile}
Assume \eqref{eq:critical-equality} and suppose that, for some \(\beta\ge0\)
and some slowly varying function \(\ell\),
\begin{equation}\label{eq:regular-profile}
\Lambda_{\mu,w}(s)\asympC s^{-\beta}\ell(s)
\qquad(s\to\infty).
\end{equation}
Then \(\gamma_w\) has finite length if and only if
\begin{equation}\label{eq:regular-profile-test}
\int^{\infty}s^{-\beta/2}\ell(s)^{1/2}\,\dd s<\infty.
\end{equation}
In particular, if \(\ell\) is bounded above and below by positive
constants for large \(s\), then the curve is accessible precisely when
\(\beta>2\).
\end{corollary}

\begin{proof}
Substitute \eqref{eq:regular-profile} into
Theorem~\ref{thm:critical-tail}.  The final assertion is the elementary
power-integral test.
\end{proof}

\subsection{Same Newton support, opposite critical lengths}

We now test whether the support function might still determine critical
length after additional compactness or discreteness assumptions.  The next
construction shows that it does not, even for countably atomic measures.

Fix \(w\in\R_{>0}^2\) and choose
\(\alpha_*\in\R_{>0}^2\) satisfying
\begin{equation}\label{eq:alpha-star-critical}
\langle w,\alpha_*\rangle=2m(w).
\end{equation}
Multiplying \(w\) by a sufficiently large positive constant only
reparametrizes the image of \(\gamma_w\).  After doing so, choose
\[
\alpha_n
=
\alpha_*
-
2^{-n}\frac{w}{|w|^2}.
\]
The rescaling may be chosen so that every \(\alpha_n\) remains in
\(\R_{\ge0}^2\).  Then
\begin{equation}\label{eq:dyadic-depth-points}
\alpha_n\longrightarrow\alpha_* ,
\qquad
\langle w,\alpha_* -\alpha_n\rangle=2^{-n}.
\end{equation}
For \(\beta>0\), define
\begin{equation}\label{eq:dyadic-measure}
\mu_\beta
=
\frac1{Z_\beta}
\sum_{n=1}^{\infty}2^{-\beta n}\delta_{\alpha_n},
\qquad
Z_\beta=\sum_{n=1}^{\infty}2^{-\beta n}.
\end{equation}
All measures \(\mu_\beta\) have the same compact support
\begin{equation}\label{eq:common-support}
E=\{\alpha_*\}\cup\{\alpha_n:n\ge1\},
\end{equation}
but \(\mu_\beta(\{\alpha_*\})=0\).

\begin{lemma}[Dyadic Laplace profile]
\label{lem:dyadic-profile}
For every \(\beta>0\),
\begin{equation}\label{eq:dyadic-profile-asymp}
\Lambda_{\mu_\beta,w}(s)\asympC s^{-\beta}
\qquad(s\ge2).
\end{equation}
\end{lemma}

\begin{proof}
The normalization constant is irrelevant to two-sided comparability, so
write
\[
S_\beta(s)=\sum_{n=1}^{\infty}2^{-\beta n}e^{-s2^{-n}}.
\]
Choose the integer \(N\) such that \(2^N\le s<2^{N+1}\).  For \(n\ge N+1\),
we have \(s2^{-n}<1\), and therefore
\[
S_\beta(s)
\ge e^{-1}\sum_{n=N+1}^{\infty}2^{-\beta n}
\asympC 2^{-\beta N}
\asympC s^{-\beta}.
\]
For the upper bound, split the sum at \(N\).  The tail satisfies
\[
\sum_{n=N+1}^{\infty}2^{-\beta n}e^{-s2^{-n}}
\le
\sum_{n=N+1}^{\infty}2^{-\beta n}
\asympC s^{-\beta}.
\]
For the head, put \(k=N-n\).  Since \(s\ge2^N\),
\[
\sum_{n=1}^{N}2^{-\beta n}e^{-s2^{-n}}
\le
2^{-\beta N}
\sum_{k=0}^{N-1}2^{\beta k}e^{-2^k}.
\]
The last series is bounded independently of \(N\), because exponential
decay in \(2^k\) dominates geometric growth.  Thus the head is also
\(O(s^{-\beta})\), proving \eqref{eq:dyadic-profile-asymp}.
\end{proof}

\begin{theorem}[Support-function insufficiency at criticality]\label{thm:support-insufficiency}
\label{thm:same-support-opposite-length}
Let \(0<\beta_-\le2<\beta_+\).  The measures
\(\mu_{\beta_-}\) and \(\mu_{\beta_+}\) defined by
\eqref{eq:dyadic-measure} have:
\begin{enumerate}[label=(\roman*)]
\item the same compact exponent support \(E\);
\item the same convex Newton body and support function;
\item zero mass on the common exposed point \(\alpha_*\);
\item opposite critical radial behavior: the curve \(\gamma_w\) has
infinite length for \(\mu_{\beta_-}\) and finite length for
\(\mu_{\beta_+}\).
\end{enumerate}
Consequently, neither the closed exponent support, the convex Newton body,
nor the support function classifies critical radial accessibility.
\end{theorem}

\begin{proof}
Items (i)--(iii) follow directly from
\eqref{eq:dyadic-measure}--\eqref{eq:common-support}.  In particular,
\(\alpha_*\) belongs to the support because every neighborhood of it
contains infinitely many atoms of positive mass, although the point itself
has zero mass.  Equation \eqref{eq:alpha-star-critical} gives the same
critical support value for all \(\beta\).  Lemma~\ref{lem:dyadic-profile}
and Corollary~\ref{cor:power-profile} show that the critical length is
finite exactly when \(\beta>2\).
\end{proof}

\begin{remark}[What has been disproved]
Theorem~\ref{thm:same-support-opposite-length} rules out any critical
classification depending only on the support set or its convex hull.  A
coefficient-sensitive invariant is essential.  The directional Laplace profile is sufficient for the weighted radial
problem by Theorem~\ref{thm:critical-tail}.  For unrestricted local
accessibility and completion, Section~\ref{sec:metric-completion} shows that
the corresponding total-degree profile is decisive.
\end{remark}

\section{Global corner accessibility and finite obstructions}
\label{sec:global}

\subsection{Relation to accessibility theory}

Global accessibility is an infimum-over-curves question rather than a
one-dimensional asymptotic calculation.  The quasihyperbolic literature
provides a classical neighboring theory in which a prescribed conformal
density is studied through lengths of arbitrary curves and the resulting
intrinsic distance \cite{GehringOsgood1979}.  Grushin geometry similarly
shows that anisotropic singular densities can produce non-Euclidean
boundary distance laws \cite{Wu2015}.  These theories supply useful
pathwise comparison principles, but the reviewed works do not formulate a
Newton support function or an exposed-face coefficient tail for an
infinite positive monomial interaction.

Curve-selection methods from semianalytic and subanalytic geometry are
another adjacent language: they often reduce local topological questions
to analytic arcs \cite{BierstoneMilman1988}.  They do not automatically
resolve the present problem.  A general compactly supported exponent
measure need not define a subanalytic conformal factor, and finite metric
length is not merely a topological incidence condition.  We therefore do
not assume that an arbitrary finite-length path can be replaced by a fixed
weighted monomial arc.

The implication from a finite radial curve to global accessibility is
immediate, whereas the converse requires pathwise lower bounds.  We first
prove the finite positive obstruction using the boundary defining function
\(\rho=x+y\).  Section~\ref{sec:metric-completion} then treats the critical
case and shows that the total-degree Laplace profile controls all approaches,
including paths whose projective direction drifts.

\begin{definition}[Global corner accessibility]
\label{def:global-accessibility}
The corner stratum \(C=\{x=y=0\}\) is locally accessible for \(g_\mu\) if
there is a locally absolutely continuous curve
\[
\gamma:[0,1)\longrightarrow U^\circ
\]
with finite \(g_\mu\)-length such that
\[
x(\gamma(t))\longrightarrow0,
\qquad
y(\gamma(t))\longrightarrow0
\]
as \(t\uparrow1\).  Otherwise the corner is locally inaccessible.
\end{definition}

\begin{proposition}[Radial criteria as global sufficient conditions]
\label{prop:radial-global-sufficient}
The corner is locally accessible if at least one of the following holds for
some \(w=(p,q)\in\R_{>0}^2\):
\begin{enumerate}[label=(\roman*)]
\item \(\cH_\mu(w)<2m(w)\);
\item \(\cH_\mu(w)=2m(w)\) and
\[
\int^\infty \Lambda_{\mu,w}(s)^{1/2}\,\dd s<\infty.
\]
\end{enumerate}
\end{proposition}

\begin{proof}
In case (i), Theorem~\ref{thm:strict-gap-radial} gives a finite-length
weighted radial curve.  In case (ii), the same conclusion follows from
Theorem~\ref{thm:critical-tail}.  Either curve satisfies
Definition~\ref{def:global-accessibility}.
\end{proof}

The converse is more delicate.  We first prove the finite positive model
needed for comparison.

\begin{theorem}[Self-contained finite Newton accessibility criterion]
\label{thm:finite-newton-self-contained}
Let
\[
 g_B=\left(\sum_{i=1}^N d_i
 x^{-\beta_{i1}}y^{-\beta_{i2}}\right)g_0,
 \qquad d_i>0,\quad \beta_i\in\R_{\ge0}^2,
\]
on a punctured codimension-two corner neighborhood, and assume
\(g_0\) is uniformly elliptic.  Put
\[
 \cH_B(w)=\max_{1\le i\le N}\langle w,\beta_i\rangle.
\]
The following are equivalent:
\begin{enumerate}[label=(\roman*)]
\item the corner is locally accessible for \(g_B\);
\item \(\max_i(\beta_{i1}+\beta_{i2})<2\);
\item there is \(w=(p,q)\in\R_{>0}^2\) such that
\[
 \cH_B(w)<2\min\{p,q\}.
\]
\end{enumerate}
Consequently, the corner is locally inaccessible exactly when
\[
 \cH_B(w)\ge2\min\{p,q\}
 \qquad\text{for every }w=(p,q)\in\R_{>0}^2.
\]
\end{theorem}

\begin{proof}
Assume first that
\(A:=\max_i(\beta_{i1}+\beta_{i2})<2\).  Along the diagonal curve
\(x=y=r\), uniform ellipticity and positivity give
\[
 |\dot\gamma(r)|_{g_B}
 \lesssim \left(\sum_{i=1}^N d_i r^{-(\beta_{i1}+\beta_{i2})}\right)^{1/2}
 \lesssim r^{-A/2}.
\]
Since \(A/2<1\), the diagonal has finite length.  Thus (ii) implies (i).

Conversely, suppose that some exponent \(\beta_j=(a,b)\) satisfies
\(a+b\ge2\).  Let \(\gamma(t)=(x(t),y(t),z(t))\) be any locally
absolutely continuous curve approaching the corner, and set
\(\rho(t)=x(t)+y(t)\).  After shrinking the neighborhood, assume
\(0<x,y<1\).  Since \(x,y\le\rho\) and \(a,b\ge0\),
\[
 x^{-a}y^{-b}\ge \rho^{-(a+b)}.
\]
Uniform ellipticity and
\(\sqrt{|\dot x|^2+|\dot y|^2}\ge |\dot\rho|/\sqrt2\) yield
\[
 L_{g_B}(\gamma)
 \ge C\int \rho(t)^{-(a+b)/2}|\dot\rho(t)|\,\dd t
\]
for a constant \(C>0\).  If \(a+b=2\), apply the chain rule to
\(-\log\rho\); if \(a+b>2\), apply it to
\(\rho^{1-(a+b)/2}/(1-(a+b)/2)\).  In either case the resulting
primitive tends to \(+\infty\) as \(\rho\to0\).  Hence every
corner-approaching curve has infinite length, proving (i) implies (ii).

It remains to compare (ii) and (iii).  If (ii) holds, choose \(w=(1,1)\);
then
\[
 \cH_B(1,1)=\max_i(\beta_{i1}+\beta_{i2})<2.
\]
Conversely, if some \(\beta_j\) has coordinate sum at least \(2\), then
for every \(w=(p,q)\in\R_{>0}^2\),
\[
 \langle w,\beta_j\rangle
 \ge\min\{p,q\}(\beta_{j1}+\beta_{j2})
 \ge2\min\{p,q\},
\]
so (iii) is impossible.  This proves all equivalences.
\end{proof}

\begin{proposition}[Finite Newton obstruction certificate]
\label{prop:finite-obstruction}
Suppose there are exponent vectors
\(\beta_1,\ldots,\beta_N\in\R_{\ge0}^2\) and constants \(d_i>0\) such that
on a punctured corner neighborhood
\begin{equation}\label{eq:finite-obstruction-lower-bound}
F_\mu(x,y,z)\ge
\sum_{i=1}^N d_i x^{-\beta_{i1}}y^{-\beta_{i2}}.
\end{equation}
Let
\[
\cH_B(w)=\max_{1\le i\le N}\langle w,\beta_i\rangle.
\]
If
\begin{equation}\label{eq:finite-obstruction-condition}
\cH_B(w)\ge2\min\{p,q\}
\qquad\text{for every }w=(p,q)\in\R_{>0}^2,
\end{equation}
then the corner is locally inaccessible for \(g_\mu\).
\end{proposition}

\begin{proof}
Let
\[
g_B=\left(\sum_{i=1}^N d_i
x^{-\beta_{i1}}y^{-\beta_{i2}}\right)g_0.
\]
Equation~\eqref{eq:finite-obstruction-lower-bound} gives
\(g_\mu\ge g_B\) as quadratic forms, hence
\(L_{g_\mu}(\gamma)\ge L_{g_B}(\gamma)\) for every curve.
Condition~\eqref{eq:finite-obstruction-condition} and
Theorem~\ref{thm:finite-newton-self-contained} imply that every
corner-approaching curve has infinite \(g_B\)-length.  The same is therefore
true for \(g_\mu\).
\end{proof}

A uniform gap in the infinite support function automatically supplies such
a finite certificate.

\begin{remark}[Finite obstruction versus the final classification]
The finite obstruction certificate above is retained because it isolates a
useful comparison mechanism.  A separate projective compactness argument can
also produce such certificates under a uniform supercritical gap.  That
intermediate result is omitted here, since the total-degree argument in
Section~\ref{sec:metric-completion} yields the stronger and exact local
completion classification directly.
\end{remark}

\section{Local metric completion}
\label{sec:metric-completion}

Accessibility only asserts the existence of one finite-length approach.  We now
address the stronger completion question: whether arbitrary interior sequences
approaching the same Euclidean corner point are Cauchy and mutually equivalent.
The first complete answer is available under a uniform total-degree gap.

Put
\begin{equation}\label{eq:max-total-degree}
 A_\mu:=\max_{\alpha\in E}(\alpha_1+\alpha_2).
\end{equation}
Because the exponent support is compact, this maximum is finite and attained.
Throughout this section, fix a relatively compact convex tangential coordinate
ball $B\Subset V$ and work in a smaller product neighborhood
\[
 U_B=(0,\varepsilon)^2\times B.
\]
The intrinsic distance induced by $g_\mu$ on $U_B$ is denoted by $d_\mu$.

\begin{definition}[Local completion fibre]
\label{def:local-completion-fibre}
Let \(\overline{U_B}^{\,d_\mu}\) be the metric completion.  By
Lemma~\ref{lem:euclidean-lower-control} below, every \(d_\mu\)-Cauchy
sequence has a unique Euclidean limit.  For
\(p\in\{x=y=0\}\times B\), the \emph{local completion fibre over \(p\)}
is the set of completion classes represented by \(d_\mu\)-Cauchy sequences
whose Euclidean limit is \(p\).
\end{definition}

\begin{lemma}[Euclidean lower control]
\label{lem:euclidean-lower-control}
After shrinking $\varepsilon$ below $1$, there is $c>0$ such that
\[
 d_\mu(P,Q)\ge c\,|P-Q|_{\mathrm E}
\]
for all $P,Q\in U_B$.
Consequently every $d_\mu$-Cauchy sequence has a unique Euclidean limit in
$[0,\varepsilon]^2\times\overline B$.
\end{lemma}

\begin{proof}
For $0<x,y<1$ and $\alpha\in E\subset\mathbb R_{\ge0}^2$ one has
$x^{-\alpha_1}y^{-\alpha_2}\ge1$.  Hence
\[
 F_\mu(x,y,z)\ge c_u\mu(E)>0.
\]
Together with the lower ellipticity bound for $g_0$, this gives
$g_\mu\ge c^2g_{\mathrm E}$ for a uniform $c>0$.  Taking infima of curve
lengths proves the distance inequality.  A $d_\mu$-Cauchy sequence is therefore
Euclidean Cauchy, and uniqueness of its Euclidean limit follows.
\end{proof}

\begin{proposition}[Marginal Lipschitz completion coordinates]
\label{prop:marginal-completion-coordinates}
For \(0<t<1\), define the marginal exponent profiles
\[
 M_x(t):=\int_E t^{-\alpha_1}\,\dd\mu(\alpha),
 \qquad
 M_y(t):=\int_E t^{-\alpha_2}\,\dd\mu(\alpha).
\]
Whenever
\[
 \int_0^\varepsilon \sqrt{M_x(t)}\,\dd t<\infty,
 \qquad
 \int_0^\varepsilon \sqrt{M_y(t)}\,\dd t<\infty,
\]
put
\[
 \Xi_x(x):=\int_0^x\sqrt{M_x(t)}\,\dd t,
 \qquad
 \Xi_y(y):=\int_0^y\sqrt{M_y(t)}\,\dd t.
\]
Then \(\Xi_x\) and \(\Xi_y\), viewed as functions on \(U_B\), are
uniformly Lipschitz with respect to \(d_\mu\).  More precisely,
\[
 |\dd\Xi_x(v)|\le \frac{1}{\sqrt{c_0c_u}}\,|v|_{g_\mu},
 \qquad
 |\dd\Xi_y(v)|\le \frac{1}{\sqrt{c_0c_u}}\,|v|_{g_\mu}.
\]
Hence both functions extend uniquely and continuously to
\(\overline{U_B}^{\,d_\mu}\).

If \(A_\mu\le2\) and
\[
 \int^\infty\sqrt{\mathcal T_\mu(S)}\,\dd S<\infty,
\]
then the two marginal integrability conditions above hold automatically.
In this regime the extensions of \(\Xi_x,\Xi_y\) vanish simultaneously
precisely over the Euclidean corner stratum.
\end{proposition}

\begin{proof}
Since \(0<y<1\) and \(\alpha_2\ge0\),
\[
 F_\mu(x,y,z)
 \ge
 c_u\int_E x^{-\alpha_1}y^{-\alpha_2}\,\dd\mu(\alpha)
 \ge
 c_u M_x(x).
\]
Together with \(g_0\ge c_0g_{\mathrm E}\), this gives
\[
 |v|_{g_\mu}
 \ge
 \sqrt{c_0c_uM_x(x)}\,|v|_{\mathrm E}
 \ge
 \sqrt{c_0c_uM_x(x)}\,|v_x|.
\]
Therefore
\[
 |\dd\Xi_x(v)|
 =
 \sqrt{M_x(x)}\,|v_x|
 \le
 \frac{1}{\sqrt{c_0c_u}}\,|v|_{g_\mu}.
\]
The proof for \(\Xi_y\) is identical.  The Lipschitz extensions to the
metric completion follow from completeness of \(\mathbb R\).

Now assume \(A_\mu\le2\).  Since
\(\alpha_1,\alpha_2\le\alpha_1+\alpha_2\) and \(0<t<1\),
\[
 M_x(t),M_y(t)
 \le
 \int_E t^{-(\alpha_1+\alpha_2)}\,\dd\mu(\alpha)
 =
 t^{-2}\mathcal T_\mu(\log(1/t)).
\]
Consequently,
\[
 \int_0^\varepsilon\sqrt{M_x(t)}\,\dd t,
 \quad
 \int_0^\varepsilon\sqrt{M_y(t)}\,\dd t
 \le
 \int_{\log(1/\varepsilon)}^\infty
 \sqrt{\mathcal T_\mu(S)}\,\dd S<\infty.
\]
Because \(\mu\neq0\), both marginal profiles are strictly positive, so
\(\Xi_x(x)=0\) exactly when \(x=0\), and similarly for \(y\).
Lemma~\ref{lem:euclidean-lower-control} gives a unique Euclidean limit to
each completion-Cauchy sequence.  Hence the two extended coordinates vanish
simultaneously exactly on completion classes whose Euclidean limit has
\(x=y=0\).
\end{proof}

\begin{remark}[Relation with explicit power coordinates]
\label{rem:power-coordinate-prototype}
For a finite two-term Liouville model
\[
 F(x,y)=x^{-a}+y^{-b},
 \qquad 0<a,b<2,
\]
the marginal coordinates above satisfy
\[
 \Xi_x(x)\asymp x^{1-a/2},
 \qquad
 \Xi_y(y)\asymp y^{1-b/2}
\]
as \(x,y\downarrow0\).  In the two-term singular Liouville model,
these explicit power coordinates are used to describe the finite metric
completion and its boundary in the publicly available companion work
\cite{ZanalAbidinLiouville2026}.  Thus they provide the atomic prototype of
the marginal exponent-measure construction above.  The present proposition
abstracts only this completion-coordinate mechanism; it does not use the
Liouville first integral, separated geodesic quadrature, global foliation, or
Hamilton--Jacobi calibration developed in that work.  The proposition is
used only as a geometric description of the completion and is not needed for
the support-and-tail classification theorem itself.
\end{remark}

\begin{lemma}[Cauchy sequence--finite curve principle]
\label{lem:cauchy-finite-curve}
Let \(p\in\{x=y=0\}\times B\).  The local completion fibre over \(p\)
is nonempty if and only if there exists an absolutely continuous curve in
\(U_B\) of finite \(g_\mu\)-length whose Euclidean limit at one endpoint is
\(p\).
\end{lemma}

\begin{proof}
A finite-length curve is Cauchy near its endpoint, so sampling it along a
sequence of endpoint parameters gives a \(d_\mu\)-Cauchy sequence converging
Euclideanly to \(p\).

Conversely, let \((P_n)\) be a \(d_\mu\)-Cauchy sequence with Euclidean
limit \(p\).  Pass to a subsequence, still denoted \((P_n)\), such that
\(d_\mu(P_n,P_{n+1})<2^{-n}\).  By the definition of intrinsic distance,
choose an absolutely continuous curve \(\gamma_n\) from \(P_n\) to
\(P_{n+1}\) with length less than \(2^{-n}+2^{-2n}\).  Concatenating the
\(\gamma_n\) gives a curve of finite total length.  Lemma~\ref{lem:euclidean-lower-control}
shows that every point on the tail of this concatenation lies Euclideanly
close to \(p\): its Euclidean distance from \(P_n\) is bounded by a
constant times the remaining metric length.  Hence the concatenated curve
converges Euclideanly to \(p\).
\end{proof}

\begin{lemma}[Subcritical corner-diameter estimate]
\label{lem:subcritical-diameter}
Assume $A_\mu<2$ and put
\[
 \theta=1-\frac{A_\mu}{2}>0.
\]
For every compact ball $B'\Subset B$ there are constants $C,r_0>0$ with the
following property.  If $z_0\in B'$ and
\[
 P=(x,y,z)\in U_B,\qquad
 R:=\max\{x,y,|z-z_0|\}<r_0,
\]
then $P$ can be joined to $(R,R,z_0)$ by a curve in $U_B$ of
$g_\mu$-length at most
\[
 C R^\theta.
\]
The constants are uniform for $z_0\in B'$.
\end{lemma}

\begin{proof}
The coefficient comparison and finiteness of $\mu$ give a constant $C_1$ such
that, whenever $0<s\le t<1$,
\begin{equation}\label{eq:ordered-coordinate-bound}
 F_\mu^0(s,t)
 =\int_Es^{-\alpha_1}t^{-\alpha_2}\,\dd\mu(\alpha)
 \le \mu(E)s^{-A_\mu},
\end{equation}
and the same estimate holds with the two coordinates interchanged.  Indeed,
$t^{-\alpha_2}\le s^{-\alpha_2}$ and
$\alpha_1+\alpha_2\le A_\mu$.  Thus
$F_\mu\le C_1s^{-A_\mu}$ on such an ordered coordinate segment.

Let $m=\min\{x,y\}$ and $M=\max\{x,y\}$.  First increase the smaller normal
coordinate from $m$ to $M$, keeping the larger coordinate and $z$ fixed.
By \eqref{eq:ordered-coordinate-bound}, uniform ellipticity, and
$A_\mu<2$, the length is bounded by
\[
 C_2\int_m^M s^{-A_\mu/2}\,\dd s
 \le C_3M^\theta\le C_3R^\theta.
\]
Next move diagonally from $(M,M,z)$ to $(R,R,z)$.  Along the diagonal,
$F_\mu(t,t)\le C_1t^{-A_\mu}$, so this segment has length at most
$C_4R^\theta$.  Finally move tangentially from $z$ to $z_0$ at the fixed
normal point $(R,R)$.  Its length is at most
\[
 C_5R^{-A_\mu/2}|z-z_0|
 \le C_5R^\theta.
\]
For $r_0$ sufficiently small the constructed path remains in $U_B$ uniformly
for $z_0\in B'$.  Summing the three estimates proves the claim.
\end{proof}

\begin{theorem}[Subcritical local metric-completion theorem]
\label{thm:subcritical-completion}
Assume
\begin{equation}\label{eq:uniform-total-subcriticality}
 A_\mu=\max_{\alpha\in E}(\alpha_1+\alpha_2)<2.
\end{equation}
Then the following hold locally along the corner stratum.
\begin{enumerate}[label=(\roman*)]
\item For every $z_0\in B$, every sequence
$P_n=(x_n,y_n,z_n)\in U_B$ converging Euclideanly to $(0,0,z_0)$ is
$d_\mu$-Cauchy.
\item Any two such sequences are equivalent in the metric completion.
Thus there is exactly one completion point, denoted $\widehat z_0$, lying over
each Euclidean corner point $(0,0,z_0)$.
\item Distinct tangential points give distinct completion points.  More
precisely, on every $B'\Subset B$ there are $c,C,r_0>0$ such that, whenever
$z,z'\in B'$ and $|z-z'|<r_0$,
\begin{equation}\label{eq:boundary-holder-comparison}
 c|z-z'|
 \le
 \widehat d_\mu(\widehat z,\widehat z')
 \le
 C|z-z'|^{1-A_\mu/2},
\end{equation}
where $\widehat d_\mu$ is the completed distance.
\item The map
\[
 z\longmapsto\widehat z
\]
is a homeomorphic embedding of $B$ onto the portion of the metric boundary
lying over the corner stratum.  In particular, no projective-direction
splitting occurs in the uniformly subcritical regime.
\end{enumerate}
\end{theorem}

\begin{proof}
Fix $z_0\in B'$ with $B'\Subset B$.  Let $P_n\to(0,0,z_0)$ Euclideanly and
set
\[
 R_n=\max\{x_n,y_n,|z_n-z_0|\}.
\]
Lemma~\ref{lem:subcritical-diameter} joins $P_n$ to
$Q_n=(R_n,R_n,z_0)$ with length at most $CR_n^\theta$.  The diagonal tail
between $Q_n$ and $Q_m$ has length bounded by
\[
 C\left|R_n^\theta-R_m^\theta\right|.
\]
Since $R_n\to0$, it follows that $(P_n)$ is $d_\mu$-Cauchy.  The same
construction applied to two sequences approaching $(0,0,z_0)$ shows that
their mutual distance tends to zero.  This proves (i) and (ii).

For the lower bound in \eqref{eq:boundary-holder-comparison}, take interior
sequences converging to $(0,0,z)$ and $(0,0,z')$ and pass the Euclidean lower
control of Lemma~\ref{lem:euclidean-lower-control} to the completion.
For the upper bound, put $R=|z-z'|$.  Approximate the two completion points by
$(r,r,z)$ and $(r,r,z')$, let $r\downarrow0$, and join each approximating
point to height $(R,R)$ using the diagonal.  Move tangentially at height
$(R,R)$, then descend diagonally.  The same estimates as in
Lemma~\ref{lem:subcritical-diameter} give total length at most $CR^\theta$.
Passing to the completion proves the upper bound.

The two-sided estimate implies injectivity and continuity of the boundary
map, as well as continuity of its inverse onto its image.  Lemma
\ref{lem:euclidean-lower-control} shows that every completion point represented
by a sequence with Euclidean limit on the corner has a uniquely determined
$z_0$, while (ii) shows that the fiber over that $z_0$ is a singleton.  This
proves (iv).
\end{proof}

\subsection{Critical total-degree tails and completion collapse}
\label{subsec:critical-completion-collapse}

The preceding theorem uses a power gap below total degree two.  At the
critical threshold the power modulus must be replaced by a tail modulus.
Assume throughout this subsection that
\begin{equation}\label{eq:non-supercritical-total-degree}
 A_\mu\le 2.
\end{equation}
Define the total-degree depth and its Laplace profile by
\begin{equation}\label{eq:total-degree-profile}
 \delta(\alpha)=2-\alpha_1-\alpha_2,
 \qquad
 \mathcal T_\mu(S)=\int_E e^{-S\delta(\alpha)}\,\dd\mu(\alpha),
 \qquad S\ge0,
\end{equation}
and put
\begin{equation}\label{eq:critical-completion-modulus}
 \Omega_\mu(r)
 =\int_{\log(1/r)}^\infty\sqrt{\mathcal T_\mu(S)}\,\dd S
  +\sqrt{\mathcal T_\mu(\log(1/r))}.
\end{equation}

\begin{lemma}[Uniform critical corner-diameter estimate]
\label{lem:critical-tail-diameter}
Suppose \eqref{eq:non-supercritical-total-degree} holds and
\begin{equation}\label{eq:critical-total-tail-integrability}
 \int^\infty\sqrt{\mathcal T_\mu(S)}\,\dd S<\infty.
\end{equation}
For every compact ball $B'\Subset B$ there are constants $C,r_0>0$ such
that, whenever $z_0\in B'$ and
\[
 P=(x,y,z)\in U_B,
 \qquad R=\max\{x,y,|z-z_0|\}<r_0,
\]
there is a curve from $P$ to $(R,R,z_0)$ of length at most
\[
 C\Omega_\mu(R).
\]
Moreover, $\Omega_\mu(R)\to0$ as $R\downarrow0$.
\end{lemma}

\begin{proof}
Let $m=\min\{x,y\}$.  If $m\le t<1$, positivity of the exponents and
\eqref{eq:non-supercritical-total-degree} give
\[
 F_\mu(x,y,z)
 \le C\int_E m^{-\alpha_1-\alpha_2}\,\dd\mu(\alpha)
 =C m^{-2}\mathcal T_\mu(\log(1/m)).
\]
The same estimate applies at every point of an ordered-coordinate segment,
with the varying smaller coordinate in place of $m$.  Raise the smaller
normal coordinate to the larger one and then move diagonally to height $R$.
The sum of the two normal costs is bounded by
\[
 C\int_0^R \frac{1}{t}
 \sqrt{\mathcal T_\mu(\log(1/t))}\,\dd t
 =C\int_{\log(1/R)}^\infty\sqrt{\mathcal T_\mu(S)}\,\dd S.
\]
At $(R,R)$, move tangentially from $z$ to $z_0$.  Since
$|z-z_0|\le R$, this costs at most
\[
 C R^{-1}\sqrt{\mathcal T_\mu(\log(1/R))}|z-z_0|
 \le C\sqrt{\mathcal T_\mu(\log(1/R))}.
\]
This proves the stated estimate.  The integral tail tends to zero by
\eqref{eq:critical-total-tail-integrability}.  The profile is nonincreasing;
therefore integrability of its square root also implies
$\mathcal T_\mu(S)\to0$, proving $\Omega_\mu(R)\to0$.
\end{proof}

\begin{theorem}[Critical-tail local metric-completion theorem]
\label{thm:critical-tail-completion}
Assume $A_\mu\le2$ and \eqref{eq:critical-total-tail-integrability}.  Then
all conclusions of Theorem~\ref{thm:subcritical-completion} remain valid,
except that the upper power bound in
\eqref{eq:boundary-holder-comparison} is replaced by
\begin{equation}\label{eq:critical-boundary-modulus}
 c|z-z'|
 \le \widehat d_\mu(\widehat z,\widehat z')
 \le C\Omega_\mu(|z-z'|)
\end{equation}
for sufficiently close $z,z'\in B'\Subset B$.  In particular, every
Euclidean corner point has exactly one completion point above it and no
projective-direction splitting occurs.
\end{theorem}

\begin{proof}
Repeat the proof of Theorem~\ref{thm:subcritical-completion}, replacing
Lemma~\ref{lem:subcritical-diameter} by
Lemma~\ref{lem:critical-tail-diameter}.  The lower estimate again follows
from Lemma~\ref{lem:euclidean-lower-control}.  For the upper estimate between
$\widehat z$ and $\widehat z'$, use the scale
$R=|z-z'|$, raise both approximating points to $(R,R)$, cross tangentially,
and descend.  The normal portions are controlled by the integral-tail term
in \eqref{eq:critical-completion-modulus}, while the tangential portion is
controlled by its second term.
\end{proof}

\begin{proposition}[Diagonal criterion]
\label{prop:diagonal-critical-equivalence}
Under $A_\mu\le2$, the diagonal curve
$\gamma(r)=(r,r,z_0)$ has finite $g_\mu$-length near the corner if and only if
\eqref{eq:critical-total-tail-integrability} holds.  Consequently, in the
non-supercritical total-degree regime, finite diagonal length automatically
upgrades to singleton-fibre metric completion as in
Theorem~\ref{thm:critical-tail-completion}.
\end{proposition}

\begin{proof}
Uniform upper and lower coefficient bounds and ellipticity give
\[
 |\dot\gamma(r)|_{g_\mu}
 \asymp r^{-1}\sqrt{\mathcal T_\mu(\log(1/r))}.
\]
The substitution $S=\log(1/r)$ proves the equivalence.  The final assertion
is Theorem~\ref{thm:critical-tail-completion}.
\end{proof}

\begin{lemma}[Path-independent total-degree lower bound]
\label{lem:path-independent-total-degree-lower}
Assume $A_\mu\le2$ and let $\gamma(t)=(x(t),y(t),z(t))$ be an absolutely
continuous curve in $U_B$ such that
\[
 R(t):=\max\{x(t),y(t)\}\longrightarrow0
\]
along one end of its parameter interval.  Then, up to a constant depending
only on the standing coefficient and ellipticity bounds,
\begin{equation}\label{eq:path-independent-tail-lower}
 L_{g_\mu}(\gamma)
 \ge
 c\int^{\infty}\sqrt{\mathcal T_\mu(S)}\,\dd S.
\end{equation}
More precisely, the contribution of every portion on which $R$ crosses from
$r_1$ to $r_0<r_1$ is bounded below by
\[
 c\int_{\log(1/r_1)}^{\log(1/r_0)}
 \sqrt{\mathcal T_\mu(S)}\,\dd S.
\]
\end{lemma}

\begin{proof}
Since $x,y\le R<1$ and the exponents are nonnegative,
\[
 x^{-\alpha_1}y^{-\alpha_2}
 \ge R^{-(\alpha_1+\alpha_2)}
 =R^{-2}e^{-\log(1/R)\delta(\alpha)}.
\]
The lower coefficient bound therefore gives
\[
 F_\mu(x,y,z)
 \ge c_0 R^{-2}\mathcal T_\mu(\log(1/R)).
\]
The function $R=\max\{x,y\}$ is absolutely continuous and satisfies
$|R'|\le\sqrt{|x'|^2+|y'|^2}\le |\gamma'|_{\mathrm E}$ almost everywhere.
Uniform lower ellipticity now yields
\[
 |\gamma'|_{g_\mu}
 \ge c_1 R^{-1}\sqrt{\mathcal T_\mu(\log(1/R))}\,|R'|.
\]
Set
\[
 \Phi(r)=\int_r^{r_1}
 \frac1t\sqrt{\mathcal T_\mu(\log(1/t))}\,\dd t.
\]
Then the metric derivative of $\Phi\circ R$ is bounded above by a constant
multiple of $|\gamma'|_{g_\mu}$.  Total variation, or equivalently the
one-dimensional area formula applied to the scale crossings of $R$, gives
\[
 L_{g_\mu}(\gamma)
 \ge c_1\left|\Phi(r_0)-\Phi(r_1)\right|.
\]
The substitution $S=\log(1/t)$ proves the stated estimate.  Letting
$r_0\downarrow0$ gives \eqref{eq:path-independent-tail-lower}.
\end{proof}

\begin{lemma}[Supercritical support obstruction]
\label{lem:supercritical-support-obstruction}
If $A_\mu>2$, every curve approaching the corner has infinite
$g_\mu$-length.
\end{lemma}

\begin{proof}
Choose $\alpha^*\in\supp\mu$ with
$\alpha^*_1+\alpha^*_2>2$.  For sufficiently small $\eta>0$, the set
\[
 E_\eta=\{\alpha\in E:|\alpha-\alpha^*|<\eta\}
\]
has positive $\mu$-measure and every $\alpha\in E_\eta$ satisfies
$\alpha_1+\alpha_2>2+\varepsilon$ for some $\varepsilon>0$.
Taking coordinatewise lower bounds on this neighborhood gives numbers
$a,b\ge0$ with $a+b>2$ and
\[
 F_\mu(x,y,z)\ge c x^{-a}y^{-b}.
\]
The one-monomial finite positive obstruction theorem
(Theorem~\ref{thm:finite-newton-self-contained}) then implies that every
corner-approaching curve has infinite length.
\end{proof}

\begin{theorem}[Complete local metric-completion classification]
\label{thm:complete-local-completion-classification}
Let $p=(0,0,z_0)$ be a point of the codimension-two corner stratum and retain
the standing positivity, compact-support, coefficient, and ellipticity
assumptions.  Exactly one of the following alternatives occurs.
\begin{enumerate}[label=(\roman*)]
\item If $A_\mu>2$, then $p$ is at infinite intrinsic distance and no point of
the metric completion lies over $p$.
\item If $A_\mu\le2$ and
\[
 \int^\infty\sqrt{\mathcal T_\mu(S)}\,\dd S=\infty,
\]
then again $p$ is at infinite intrinsic distance and the completion fibre over
$p$ is empty.
\item If $A_\mu\le2$ and
\[
 \int^\infty\sqrt{\mathcal T_\mu(S)}\,\dd S<\infty,
\]
then the completion fibre over $p$ consists of exactly one point.  Every
interior sequence converging Euclideanly to $p$ is $d_\mu$-Cauchy, and any two
such sequences are equivalent.
\end{enumerate}
Consequently, under the standing assumptions, nonempty split completion
fibres do not occur.  Equivalently, local corner completion is completely
decided by the supercritical-support obstruction and, in the
non-supercritical case, by the total-degree Laplace-tail integral.
\end{theorem}

\begin{proof}
Lemma~\ref{lem:supercritical-support-obstruction} shows that every curve
approaching \(p\) has infinite length; Lemma~\ref{lem:cauchy-finite-curve}
then shows that the completion fibre is empty, proving (i).
Under $A_\mu\le2$, Lemma~\ref{lem:path-independent-total-degree-lower} shows
that divergence of the tail integral forces every corner-approaching curve to
have infinite length.  A second application of
Lemma~\ref{lem:cauchy-finite-curve} proves that the fibre is empty, proving
(ii).  If the integral converges,
Theorem~\ref{thm:critical-tail-completion} proves that all Euclidean approaches
are Cauchy and mutually equivalent, giving a singleton fibre and proving
(iii).  The three alternatives are mutually exclusive and exhaustive.
\end{proof}

\begin{corollary}[Accessibility, Cauchy approach, and completion collapse]
\label{cor:accessibility-completion-equivalence}
When $A_\mu\le2$, the following are equivalent:
\begin{enumerate}[label=(\alph*)]
\item some curve reaches the corner with finite length;
\item the diagonal curve has finite length;
\item $\int^\infty\sqrt{\mathcal T_\mu(S)}\,\dd S<\infty$;
\item every Euclidean approach sequence is Cauchy and all such sequences define
one completion point.
\end{enumerate}
\end{corollary}

\begin{proof}
The implication (b)$\Leftrightarrow$(c) is
Proposition~\ref{prop:diagonal-critical-equivalence}, and
(c)$\Rightarrow$(d) is
Theorem~\ref{thm:critical-tail-completion}.  The implication (d)$\Rightarrow$(a) follows by sampling the diagonal and applying Lemma~\ref{lem:cauchy-finite-curve}.
Finally, (a)$\Rightarrow$(c) follows from
Lemma~\ref{lem:path-independent-total-degree-lower}.
\end{proof}

\begin{proposition}[Anisotropic stress test]
\label{prop:anisotropic-stress-test}
Assume \(A_\mu\le2\) and
\[
\int^\infty\sqrt{\mathcal T_\mu(S)}\,\dd S<\infty.
\]
Let \(P_n=(x_n,y_n,z_n)\) and
\(Q_n=(\tilde x_n,\tilde y_n,\tilde z_n)\) be arbitrary sequences
converging Euclideanly to the same corner point.  No bound is imposed on
\(x_n/y_n\) or \(\tilde x_n/\tilde y_n\).  Then
\[
 d_\mu(P_n,Q_n)\longrightarrow0.
\]
Thus even exponentially or super-exponentially drifting projective directions
do not generate distinct completion classes.
\end{proposition}

\begin{proof}
Put
\[
 R_n=\max\{x_n,y_n,\tilde x_n,\tilde y_n,
 |z_n-z_0|,|\tilde z_n-z_0|\}.
\]
Apply Lemma~\ref{lem:critical-tail-diameter} to join both points to
\((R_n,R_n,z_0)\).  The two resulting lengths are bounded by
\(C\Omega_\mu(R_n)\), which tends to zero.
\end{proof}

\begin{example}[A genuinely critical completion point]
\label{ex:critical-diffuse-completion}
Let $\beta>0$ and let $\mu_\beta$ be the pushforward of
$\beta t^{\beta-1}\,\dd t$ on $(0,1]$ under
\[
 t\longmapsto (2-t,0).
\]
Its compact support contains $(2,0)$, so $A_{\mu_\beta}=2$, but it has no atom
on the critical face.  Its total-degree profile is
\[
 \mathcal T_{\mu_\beta}(S)
 =\beta\int_0^1e^{-St}t^{\beta-1}\,\dd t
 \sim \beta\Gamma(\beta)S^{-\beta}.
\]
Thus \eqref{eq:critical-total-tail-integrability} holds exactly when
$\beta>2$.  For $\beta>2$, Theorem~\ref{thm:critical-tail-completion} gives
one completion point over every point of the corner stratum, despite the
absence of a strict total-degree gap.
\end{example}

\begin{example}[A critical atom remains infinitely far away]
\label{ex:critical-atom-no-completion}
For $\mu=\delta_{(1,1)}$ one has $A_\mu=2$ and
$\mathcal T_\mu(S)\equiv1$.  The diagonal integral diverges.  In fact the
finite positive obstruction theorem applies to the monomial metric
$(xy)^{-1}g_0$ and shows that the corner is at infinite distance.  Hence the
condition $A_\mu=2$ alone permits both completion collapse
(Example~\ref{ex:critical-diffuse-completion} with $\beta>2$) and complete
inaccessibility.
\end{example}

\begin{remark}[Why drifting directions do not create an escape]
\label{rem:critical-completion-open}
The apparent mixed or drifting-direction problem is removed by positivity.
If $A_\mu>2$, a positive-mass neighborhood of one supercritical support point
already supplies a finite monomial lower obstruction.  If $A_\mu\le2$, the
scale $R=\max\{x,y\}$ gives a path-independent lower bound involving the same
total-degree tail that controls the diagonal.  Thus changing projective
direction cannot evade a divergent tail, and no additional projective family
of completion invariants is needed for the local fibre classification.
\end{remark}

\section{Finite truncations and monotone stability}
\label{sec:truncations}

This section studies countably atomic interactions
\begin{equation}\label{eq:atomic-full-measure}
\mu=\sum_{j=1}^{\infty}c_j\delta_{\alpha_j},
\qquad c_j>0,
\qquad \sum_{j=1}^{\infty}c_j<\infty,
\end{equation}
with compact exponent closure, together with their finite truncations
\begin{equation}\label{eq:truncated-measure}
\mu_N=\sum_{j=1}^{N}c_j\delta_{\alpha_j}.
\end{equation}
Write \(F_N=F_{\mu_N}^0\), \(F=F_\mu^0\), and
\(g_N=F_Ng_0\), \(g=Fg_0\).

\subsection{Relation to monotone metric convergence}

Monotone sequences of metric tensors and their induced distance functions
form an established topic in metric convergence.  In particular, increasing
Riemannian metrics induce increasing intrinsic distances, and under global
compactness and diameter hypotheses stronger uniform and
Gromov--Hausdorff conclusions are available.  The present local singular
problem is narrower in one direction and more delicate in another: the
basic fixed-curve convergence follows directly from the monotone
convergence theorem, while convergence of the infimum over all curves need
not follow without additional compactness or localization assumptions.

The results below are therefore not presented as a new abstract convergence
theory for Riemannian distances.  Their paper-specific content is the
interaction between finite Newton truncation, projective support functions,
and critical accessibility.  In particular, the section identifies the
one-sided instability permitted by positivity: accessibility can be lost in
the infinite critical limit, but it cannot be gained by adding positive
terms.

\begin{proposition}[Uniform convergence away from the corner]
\label{prop:uniform-convergence-interior}
Let \(K\Subset U^\circ\).  Then \(F_N\to F\) uniformly on \(K\).  In
particular, \(g_N\to g\) uniformly as quadratic forms on \(K\).
\end{proposition}

\begin{proof}
Choose \(\eta>0\) such that \(x,y\ge\eta\) on \(K\).  Compactness of the
exponent closure gives finite numbers \(A_0,B_0\) with
\(0\le a_j\le A_0\) and \(0\le b_j\le B_0\).  Hence
\[
0\le F-F_N
=\sum_{j>N}c_jx^{-a_j}y^{-b_j}
\le \eta^{-A_0-B_0}\sum_{j>N}c_j
\]
on \(K\), and the right-hand side tends to zero.
\end{proof}

\begin{proposition}[Monotone convergence of fixed-curve lengths]
\label{prop:fixed-curve-monotone}
Let \(\gamma:(a,b)\to U^\circ\) be an absolutely continuous curve.  Then
\begin{equation}\label{eq:length-monotone-limit}
L_{g_N}(\gamma)\uparrow L_g(\gamma)
\end{equation}
in \([0,\infty]\).
\end{proposition}

\begin{proof}
The conformal factors satisfy \(F_N\uparrow F\) pointwise.  Therefore the
nonnegative length densities
\[
\sqrt{F_N(\gamma(t))}\,|\dot\gamma(t)|_{g_0}
\]
increase pointwise to the corresponding \(g\)-length density.  The monotone
convergence theorem gives \eqref{eq:length-monotone-limit}.
\end{proof}

\begin{corollary}[One-sided stability of accessibility]
\label{cor:one-sided-accessibility}
For every fixed curve \(\gamma\):
\begin{enumerate}[label=(\roman*)]
\item if \(L_g(\gamma)<\infty\), then
      \(L_{g_N}(\gamma)<\infty\) for every \(N\) and
      \(L_{g_N}(\gamma)\to L_g(\gamma)\);
\item if \(L_{g_N}(\gamma)=\infty\) for some \(N\), then
      \(L_g(\gamma)=\infty\);
\item it is possible that \(L_{g_N}(\gamma)<\infty\) for every \(N\) but
      \(L_g(\gamma)=\infty\).
\end{enumerate}
Consequently, adding positive Newton terms can destroy accessibility of a
prescribed curve, but cannot create it.
\end{corollary}

\begin{proof}
Parts (i) and (ii) follow immediately from
Proposition~\ref{prop:fixed-curve-monotone}.  Part (iii) is realized by
Example~\ref{ex:critical-loss-under-truncation} below.
\end{proof}

\subsection{Convergence of finite support functions}

Define
\begin{equation}\label{eq:truncated-support-functions}
\cH_N(w)=\max_{1\le j\le N}\langle w,\alpha_j\rangle.
\end{equation}

\begin{proposition}[Uniform projective convergence]
\label{prop:uniform-support-convergence}
Let
\[
\Sigma=\{(p,q)\in\R_{\ge0}^2:p+q=1\}.
\]
Then
\[
\cH_N\uparrow\cH_\mu
\]
pointwise and uniformly on \(\Sigma\).
\end{proposition}

\begin{proof}
The functions \(\cH_N\) are continuous and increase pointwise to
\[
\sup_{j\ge1}\langle w,\alpha_j\rangle.
\]
Since the support of the atomic measure is the closure of
\(\{\alpha_j:j\ge1\}\), this supremum equals \(\cH_\mu(w)\).  The limit is
continuous on the compact set \(\Sigma\).  Dini's theorem therefore gives
uniform convergence.
\end{proof}

\begin{corollary}[Stability under a strict projective gap]
\label{cor:strict-gap-truncation}
Let \(K\subset\Sigma\) be compact.
\begin{enumerate}[label=(\roman*)]
\item If
\[
\cH_\mu(w)\le 2\min\{p,q\}-\delta
\qquad(w=(p,q)\in K)
\]
for some \(\delta>0\), then the same inequality with \(\delta/2\) holds
for all sufficiently large \(N\).
\item If
\[
\cH_\mu(w)\ge 2\min\{p,q\}+\delta
\qquad(w=(p,q)\in K),
\]
then the same inequality with \(\delta/2\) holds for all sufficiently large
\(N\).
\end{enumerate}
\end{corollary}

\begin{proof}
Apply uniform convergence from
Proposition~\ref{prop:uniform-support-convergence}.
\end{proof}

The corollary shows that truncation is stable away from the critical set.
All genuinely new instability is therefore concentrated where the limiting
support function meets the metric threshold.

\subsection{Critical loss of accessibility}

\begin{example}[Every truncation accessible, critical limit inaccessible]
\label{ex:critical-loss-under-truncation}
Fix a weighted direction \(w\) and choose exponents \(\alpha_n\) satisfying
\begin{equation}\label{eq:critical-approach-exponents}
\langle w,\alpha_n\rangle
=2m(w)-2^{-n},
\qquad m(w)=\min\{p,q\}.
\end{equation}
Let
\begin{equation}\label{eq:critical-approach-weights}
c_n=Z_\beta^{-1}2^{-\beta n},
\qquad
Z_\beta=\sum_{n=1}^{\infty}2^{-\beta n},
\qquad \beta>0.
\end{equation}
For every finite truncation,
\[
\cH_N(w)=2m(w)-2^{-N}<2m(w),
\]
so the weighted radial curve \(\gamma_w\) has finite \(g_N\)-length.
For the infinite measure,
\[
\cH_\mu(w)=2m(w).
\]
The critical-tail calculation of
Section~\ref{sec:critical} gives
\[
\Lambda_{\mu,w}(s)\asymp s^{-\beta}.
\]
Consequently,
\[
L_g(\gamma_w)<\infty
\quad\Longleftrightarrow\quad
\beta>2.
\]
In particular, when \(0<\beta\le2\), every finite truncation is radially
accessible in direction \(w\), whereas the infinite limit is not.
\end{example}

\begin{remark}[Correction of the opposite scenario]
\label{rem:impossible-reverse-truncation}
The reverse transition cannot occur for positive truncations.  It is
impossible for some finite truncation to give infinite length to a fixed
curve while the full infinite interaction gives that curve finite length,
because \(g_N\le g\) and hence
\(L_{g_N}(\gamma)\le L_g(\gamma)\).  Thus the earlier heuristic possibility
that finite truncations might be inaccessible while the infinite positive
limit becomes accessible was incorrect.
\end{remark}

\subsection{Intrinsic distances: a necessary warning}

Let \(d_N\) and \(d\) denote the intrinsic distances induced by \(g_N\) and
\(g\) on \(U^\circ\).  Positivity gives
\begin{equation}\label{eq:distance-monotonicity}
d_N(x_1,x_2)\le d_{N+1}(x_1,x_2)\le d(x_1,x_2).
\end{equation}
Hence \(\bar d=\sup_Nd_N\) exists and satisfies \(\bar d\le d\).  Equality
is not asserted here.  Proposition~\ref{prop:fixed-curve-monotone} concerns
a fixed curve, whereas intrinsic distance takes an infimum over curves, and
in general one cannot interchange the monotone limit with that infimum
without an additional compactness or localization argument.  This is the
precise point at which abstract monotone metric-convergence theory becomes
adjacent but does not automatically settle the present singular local
problem.

\section{Directional classification data and limits of uniqueness}
\label{sec:classification}

\subsection{Relation to inverse-transform questions}

The pair consisting of a support function and a Laplace transform has close
relatives in several established theories.  The support-function limit is a
max-plus or tropical dequantization: it retains the largest weighted exponent
while discarding subleading coefficient information
\cite{Litvinov2013}.  Classical uniqueness theorems for Laplace--Stieltjes
transforms show that a complete one-dimensional Laplace profile determines its
depth pushforward measure.  In a different direction, Cram\'er--Wold type
results study when a multivariate measure is determined by its one-dimensional
projections \cite{BelisleMasseRansford1997}.  These theories concern transform
or measure reconstruction, not metric accessibility.

The comparison changes the appropriate strength of the main claim.  For a fixed
weighted ray, the data \((\cH_\mu(w),\Lambda_{\mu,w})\) determine the frozen
conformal factor exactly, hence they certainly determine its length.  However,
this pair is not a minimal accessibility invariant: away from criticality only
the sign of \(\cH_\mu(w)-2m(w)\) matters, and at criticality only convergence or
divergence of one scalar integral is needed.  Nor does one directional profile
identify the original exponent measure, because it records only the
pushforward under the depth map.  Accordingly, the theorem below is called a
complete \emph{weighted-radial classification}, not a unique classification of
measures or of global corner geometry.

\begin{theorem}[Complete weighted-radial classification]
\label{thm:complete-radial-classification}
Let \(w=(p,q)\in\R_{>0}^2\), put \(m(w)=\min\{p,q\}\), and let
\(\gamma_w(r)=(r^p,r^q,z_0)\).  Then exactly one of the following alternatives
holds:
\begin{enumerate}[label=(\roman*)]
\item If \(\cH_\mu(w)<2m(w)\), then \(L_{g_\mu}(\gamma_w)<\infty\).
\item If \(\cH_\mu(w)>2m(w)\), then \(L_{g_\mu}(\gamma_w)=\infty\).
\item If \(\cH_\mu(w)=2m(w)\), then
\[
L_{g_\mu}(\gamma_w)<\infty
\quad\Longleftrightarrow\quad
\int^{\infty}\Lambda_{\mu,w}(s)^{1/2}\,\dd s<\infty.
\]
\end{enumerate}
Thus weighted-radial accessibility is completely determined by the support
order away from equality and by the exposed-face Laplace tail at equality.
\end{theorem}

\begin{proof}
The subcritical and supercritical alternatives are
Theorem~\ref{thm:strict-gap-radial}.  The equality alternative is
Theorem~\ref{thm:critical-tail}.  The three cases are mutually exclusive and
exhaust all possible values of \(\cH_\mu(w)-2m(w)\).
\end{proof}

\begin{corollary}[Four-valued radial decision datum]
\label{cor:decision-invariant}
For a fixed direction \(w\), define
\[
\mathfrak R_\mu(w)\in
\{\mathrm{subcritical},\mathrm{supercritical},
\mathrm{critical\mbox{-}finite},\mathrm{critical\mbox{-}infinite}\}
\]
by
\[
\mathfrak R_\mu(w)=
\begin{cases}
\mathrm{subcritical},
 &\cH_\mu(w)<2m(w),\\
\mathrm{supercritical},
 &\cH_\mu(w)>2m(w),\\
\mathrm{critical\mbox{-}finite},
 &\begin{aligned}
   &\cH_\mu(w)=2m(w),\\[-2pt]
   &\displaystyle\int^\infty\sqrt{\Lambda_{\mu,w}(s)}\,\dd s<\infty,
  \end{aligned}\\
\mathrm{critical\mbox{-}infinite},
 &\begin{aligned}
   &\cH_\mu(w)=2m(w),\\[-2pt]
   &\displaystyle\int^\infty\sqrt{\Lambda_{\mu,w}(s)}\,\dd s=\infty.
  \end{aligned}
\end{cases}
\]
Then \(\gamma_w\) has finite length exactly in the subcritical and
critical-finite regimes.  Thus the full function \(\Lambda_{\mu,w}\) is
sufficient but not minimal for the binary radial accessibility decision.
\end{corollary}

\begin{proof}
This is an immediate restatement of
Theorem~\ref{thm:complete-radial-classification}.
\end{proof}

\begin{proposition}[One directional profile does not identify the measure]
\label{prop:profile-nonuniqueness}
Fix \(w\in\R_{>0}^2\).  There exist distinct compactly supported finite
positive measures \(\mu_1\ne\mu_2\) on \(\R_{\ge0}^2\) such that
\[
\cH_{\mu_1}(w)=\cH_{\mu_2}(w)
\quad\text{and}\quad
\Lambda_{\mu_1,w}(s)=\Lambda_{\mu_2,w}(s)
\quad(s\ge0).
\]
Consequently, the directional pair classifies the pullback along
\(\gamma_w\), but not the full exponent measure.
\end{proposition}

\begin{proof}
Choose a point \(\alpha_*\in\R_{>0}^2\) and two distinct points
\(\alpha,\widetilde\alpha\in\R_{\ge0}^2\) lying on the same affine level set
of \(\langle w,\cdot\rangle\), strictly below \(\alpha_*\):
\[
\langle w,\alpha\rangle
=
\langle w,\widetilde\alpha\rangle
<
\langle w,\alpha_*\rangle.
\]
For positive constants \(a,b\), set
\[
\mu_1=a\delta_{\alpha_*}+b\delta_{\alpha},
\qquad
\mu_2=a\delta_{\alpha_*}+b\delta_{\widetilde\alpha}.
\]
The two measures are distinct.  Their support values in direction \(w\) are
both \(\langle w,\alpha_*\rangle\), and their depth pushforwards are both
\[
a\delta_0+b\delta_{\langle w,\alpha_*\rangle-
\langle w,\alpha\rangle}.
\]
Hence their directional Laplace profiles coincide for every \(s\ge0\).
\end{proof}

\begin{example}[Different critical profiles with the same decision]
\label{ex:profile-not-minimal}
In the dyadic family of Section~\ref{sec:critical}, any two parameters
\(\beta_1,\beta_2>2\) give different asymptotic profiles
\(\Lambda_{\mu_{\beta_i},w}(s)\asymp s^{-\beta_i}\), while both yield finite
critical length.  Likewise, any two parameters in \((0,2]\) yield different
profiles but the same divergent decision.  Thus accessibility does not recover
the complete profile even within a fixed-support atomic family.
\end{example}

\begin{remark}[Fixed-ray data versus the full corner fibre]
Theorem~\ref{thm:complete-radial-classification} classifies each prescribed
weighted ray by its directional profile.  The full local corner fibre is
instead governed by the total-degree profile through
Theorem~\ref{thm:complete-local-completion-classification}.  No claim is made
that either the binary accessibility decision or the full family of such
decisions uniquely reconstructs the exponent measure.
\end{remark}

\section{Sharpness and the role of the hypotheses}
\label{sec:sharpness}

The completion theorem is formulated for a deliberately uniform positive
class.  This section separates assumptions that are structurally indispensable
for a classification in terms of the exponent measure from assumptions that
mainly provide a convenient closed framework.  The counterexamples are local
and use the Euclidean background unless stated otherwise.

\subsection{Two-sided coefficient comparison is essential}

The measure-only invariant cannot classify the metric when the coefficient is
allowed either to vanish or to blow up at the corner.

\begin{proposition}[Failure without the uniform lower coefficient bound]
\label{prop:no-lower-bound}
Let $E=\{(2,0)\}$, $\mu=\delta_{(2,0)}$, and
\[
 u((2,0);x,y,z)=x^2.
\]
Then $A_\mu=2$ and $\mathcal T_\mu(S)\equiv1$, so the tail integral associated
with the exponent measure diverges.  Nevertheless
\[
 F_\mu(x,y,z)=x^2x^{-2}=1,
\]
and the local metric completion is the ordinary Euclidean closure.  In
particular, every corner point occurs at finite distance.
\end{proposition}

\begin{proof}
The displayed identity reduces $g_\mu$ to $g_0$.  The conclusion follows from
uniform ellipticity of $g_0$.
\end{proof}

Thus the lower bound in \eqref{eq:coefficient-bounds} is not merely used to
simplify estimates: it prevents a coefficient from erasing the singular order
recorded by the support.

\begin{proposition}[Failure without the uniform upper coefficient bound]
\label{prop:no-upper-bound}
Let $E=\{(0,0)\}$, $\mu=\delta_{(0,0)}$, and, on a sufficiently small corner
chart, set
\[
 u((0,0);x,y,z)=\frac{1}{x^2+y^2}.
\]
Then $A_\mu=0$ and
\[
 \int^\infty \sqrt{\mathcal T_\mu(S)}\,\dd S
 =\int^\infty e^{-S}\,\dd S<\infty,
\]
but every curve approaching the corner has infinite length.  Hence the
completion fibre is empty.
\end{proposition}

\begin{proof}
Here $F_\mu=(x^2+y^2)^{-1}$.  Put $r=(x^2+y^2)^{1/2}$.  Uniform lower
ellipticity and $|r'|\le |\gamma'|_{\mathrm E}$ give
\[
 |\gamma'|_{g_\mu}\ge c\frac{|r'|}{r}.
\]
Every curve with $r\to0$ therefore has length at least
$c\int_0^{r_0}\dd r/r=\infty$.  Lemma~\ref{lem:cauchy-finite-curve} then
implies that the completion fibre is empty.
\end{proof}

Together, Propositions~\ref{prop:no-lower-bound} and
\ref{prop:no-upper-bound} show that the exponent measure determines the
completion only up to a uniform two-sided multiplicative comparison.  A more
general theory would have to include the asymptotic order of $u$ as additional
data.

\subsection{The nonnegative exponent cone prevents boundary collapse}

The restriction $E\subset\mathbb R_{\ge0}^2$ is used both in the common-scale
upper construction and in Euclidean lower control.  Allowing negative
exponents can make the conformal factor vanish and can identify distinct
Euclidean boundary points.

\begin{proposition}[Collapse when negative exponents are admitted]
\label{prop:negative-exponent-collapse}
Let $E=\{(-2,0)\}$, $\mu=\delta_{(-2,0)}$, $u\equiv1$, and $g_0=g_{\mathrm E}$.
Then
\[
 g_\mu=x^2g_{\mathrm E}.
\]
For any two tangential points $z,z'$ in a fixed compact ball, the sequences
\[
 P_n=(n^{-1},n^{-1},z),\qquad
 Q_n=(n^{-1},n^{-1},z')
\]
are Cauchy and satisfy $d_\mu(P_n,Q_n)\to0$.  Hence distinct Euclidean corner
points represent the same completion point, and the unique-Euclidean-limit
property in Lemma~\ref{lem:euclidean-lower-control} fails.
\end{proposition}

\begin{proof}
At fixed $x=n^{-1}$ and $y=n^{-1}$, the straight tangential segment has length
$n^{-1}|z-z'|$.  This proves $d_\mu(P_n,Q_n)\to0$.  Each sequence is Cauchy because the diagonal segment
$t\mapsto(t,t,z)$ from $t=n^{-1}$ to $t=m^{-1}$ has length
$\sqrt2\int t\,\dd t=2^{-1/2}|n^{-2}-m^{-2}|$.  Thus the two sequences define the same completion class
although their Euclidean limits are different.
\end{proof}

This example does not say that metrics with vanishing conformal factors are
uninteresting.  It shows that they require a different boundary map: the
completion need no longer project injectively to the Euclidean corner stratum.

\subsection{Background ellipticity is structurally necessary}

The proofs use only bi-Lipschitz comparison with a smooth nondegenerate
background metric.  If the background degenerates, it can cancel the Newton
singularity completely.

\begin{proposition}[Failure for a degenerate background]
\label{prop:degenerate-background}
Let $E=\{(2,2)\}$, $\mu=\delta_{(2,2)}$, $u\equiv1$, and take on the interior
\[
 g_0=x^2y^2g_{\mathrm E}.
\]
Then $A_\mu=4>2$, whereas
\[
 g_\mu=x^{-2}y^{-2}g_0=g_{\mathrm E}.
\]
Thus the corner lies at finite distance despite the supercritical exponent.
\end{proposition}

The example lies outside the standing class because $g_0$ is not uniformly
elliptic up to the corner.  It shows that the theorem is invariant under
uniform bi-Lipschitz changes of background, but not under singular or
degenerate changes.

\subsection{Positivity and canonical support}

Positivity has two distinct roles.  First, it guarantees that $F_\mu$ is a
Riemannian conformal factor.  Second, it prevents cancellation, which is what
turns support neighborhoods into lower metric bounds.  With signed or
complex coefficients, a raw list of exponents is not canonical: equal leading
orders may cancel before the metric is formed.  Consequently the support
function of an unreduced signed representation cannot be an intrinsic metric
invariant.  A signed theory would first need a canonical noncancelling leading
object and separate hypotheses guaranteeing $F>0$.  The present
empty-or-singleton theorem is therefore genuinely a positive-interaction
result.

\subsection{Compactness and regularity are convenient, not fully sharp}

Compact exponent support is sufficient for all of the following at once:
interior finiteness, attainment and continuity of support functions, uniform
ray estimates, and positive mass near a maximizing support point.  It is not
claimed to be necessary.  For example, an unbounded atomic support with
super-exponentially decaying masses can still define a finite interior
interaction.  The arguments suggest that compactness could be replaced by a
collection of exponential-moment and support-attainment conditions adapted to
the directions under study.  Such a reformulation would lengthen the
hypotheses without changing the geometric mechanism, so it is not pursued
here.

Likewise, continuity of $F_\mu$ is a convenient condition ensuring a standard
Riemannian length metric.  Most estimates only require a positive Borel density
with suitable local bounds and an intrinsic length structure.  No sharp
regularity threshold is asserted.

The finiteness and nontriviality of $\mu$ are normalization-level assumptions.
Nontriviality is necessary for $F_\mu>0$.  Finiteness can be weakened whenever
the exponent integral is locally finite and the comparison estimates remain
valid, but a finite measure makes the model and its Laplace profiles
unambiguous.

\subsection{Sharpness status}

The role of the standing assumptions can be summarized as follows.
\begin{center}
\begin{tabular}{p{0.29\textwidth}p{0.24\textwidth}p{0.37\textwidth}}
\toprule
Hypothesis & Status in this paper & Reason \\
\midrule
Positive measure and coefficients & Structural & Prevents cancellation and supplies all lower bounds. \\
Uniform lower bound on $u$ & Essential for a measure-only criterion & Proposition~\ref{prop:no-lower-bound}. \\
Uniform upper bound on $u$ & Essential for a measure-only criterion & Proposition~\ref{prop:no-upper-bound}. \\
Nonnegative exponents & Structural for the stated boundary projection & Proposition~\ref{prop:negative-exponent-collapse}. \\
Uniform ellipticity of $g_0$ & Essential up to bi-Lipschitz replacement & Proposition~\ref{prop:degenerate-background}. \\
Compact exponent support & Sufficient, not claimed sharp & May be replaced by directional moment and attainment hypotheses. \\
Continuity of $F_\mu$ & Technical regularity & Measurable length-density variants should be possible. \\
Finite nonzero $\mu$ & Convenient normalization & Local finiteness and positivity are the true requirements. \\
\bottomrule
\end{tabular}
\end{center}

Accordingly, the sharp claim of the paper is not that every assumption is
minimal.  It is that within the natural positive, uniformly comparable Newton
class, the support-plus-tail criterion is exact, and that the hypotheses which
make this criterion representation-independent cannot be discarded without
adding new asymptotic data.

\section{Scope of the completion theorem}

Theorem~\ref{thm:complete-local-completion-classification} is a local theorem
for one codimension-two corner chart.  It does not by itself identify the
global completion of an arbitrary manifold with corners: distinct charts or
strata may interact through shorter curves outside the chosen neighborhood.
The local conclusion is nevertheless intrinsic in the following precise
sense.  Within a sufficiently small product neighborhood, the fibre over each
corner point is empty or a singleton, and the alternative is decided by the
total-degree support and tail criteria stated above.

The theorem also depends essentially on positivity and uniform two-sided
coefficient comparison.  Signed cancellation, coefficients that vanish on
selected scales, or metrics without uniform background ellipticity fall
outside the classification.  These restrictions are part of the theorem,
not merely technical conveniences.

\section{Outlook: beyond exact exponent-measure models}
\label{sec:outlook-comparability}

The exact exponent-measure representation used throughout the paper is
structurally restrictive, but the completion classification is stable under
uniform two-sided comparison. This observation enlarges the geometric scope
of the theory without changing its principal invariant.

\begin{definition}[Newton-comparable density]
Let \(F_\mu\) be a positive Newton interaction as in
Definition~\ref{def:measure-interaction}. A positive density
\(\rho\) is \emph{locally Newton-comparable} to \(F_\mu\) near the corner if
there are constants \(0<c\le C<\infty\) such that
\begin{equation}\label{eq:newton-comparable}
cF_\mu(x,y,z)\le \rho(x,y,z)\le CF_\mu(x,y,z)
\end{equation}
throughout a sufficiently small punctured product neighborhood.
\end{definition}

\begin{proposition}[Transfer under Newton comparability]
\label{prop:newton-comparable-transfer}
Let
\[
g_\rho=\rho g_0,
\qquad
g_\mu=F_\mu g_0,
\]
and assume \eqref{eq:newton-comparable}. Then, for the local path metrics
defined using curves contained in the comparison neighborhood,
\[
\sqrt c\,d_\mu\le d_\rho\le \sqrt C\,d_\mu.
\]
Consequently, the identity on the punctured neighborhood extends to a
bi-Lipschitz bijection of the corresponding local metric completions.
In particular, the fibre over a prescribed Euclidean corner point is empty
for \(g_\rho\) if and only if it is empty for \(g_\mu\), and it is a singleton
for \(g_\rho\) if and only if it is a singleton for \(g_\mu\).
\end{proposition}

\begin{proof}
For every absolutely continuous curve \(\gamma\),
\[
\sqrt c\,L_{g_\mu}(\gamma)
\le
L_{g_\rho}(\gamma)
\le
\sqrt C\,L_{g_\mu}(\gamma).
\]
Taking infima over the same curve family gives the distance comparison.
The identity and its inverse therefore preserve Cauchy sequences and
equivalence of Cauchy sequences, so they extend to mutually inverse
bi-Lipschitz maps of the local completions.
\end{proof}

\begin{remark}
Proposition~\ref{prop:newton-comparable-transfer} is a standard
bi-Lipschitz transfer principle and is not an additional principal theorem
of the paper. Its importance is that it shifts the extension problem from
metric geometry to an analytic recognition problem.
\end{remark}

\subsection{Why exponent measures are useful models}

The exponent-measure formalism unifies finite, countably atomic, and diffuse
positive interactions. A finite Newton sum corresponds to an atomic
measure, while an infinite or diffuse interaction is represented in the same
exponent space. The support separates the leading convex geometry from the
distribution of mass: the support function controls noncritical power order,
whereas the mass near the total-degree critical face controls the Laplace
tail and hence the completion transition.

This separation is unavailable for a completely arbitrary conformal
density. It is the principal reason to regard exponent-measure models as
candidate asymptotic normal forms, while making no claim that every positive
density admits such a form.

\subsection{Analytic routes to recognition}

Several established theories approach the recognition problem from different
sides.  Generalized analytic and generalized power-series monomialization can
produce local monomial-type forms after suitable blow-ups
\cite{MartinRolinSanz2013,MolinaPalmaSanz2024,Palma2024}.  Such normal forms
are promising sources of one-atom or finite Newton models, but the induced
metric also contains the differential of the blow-down map, and chartwise
models must still be glued.  The existing resolution results therefore do not
by themselves imply the completion classification for the original metric.

A finite positive Newton principal part with a remainder already satisfying
\(|R|\le\theta P\), \(\theta<1\), is automatically comparable to its
principal part.  This observation is useful but essentially tautological: the
hypothesis has already imposed the two-sided estimate that a recognition
theorem is meant to derive.

Multivariable regular variation provides scaling limits for functions and
measures and separates radial growth from angular data
\cite{MeerschaertScheffler1999,HultLindskog2006}.  Directional asymptotics
alone, however, do not imply a single uniform comparison on a full punctured
corner neighborhood.  A Newton-recognition theorem based on regular variation
would require compatible angular limits, uniform projective control, and
bounds on the slowly varying factors.

Finally, Harnack and volume-growth conditions are deliberately broad.  They
support accessibility and metric-boundary results for general conformal
densities \cite{Nieminen2009,KlenSuomala2013}, but do not force monomial,
log-convex, or exponent-measure structure.  They are therefore insufficient
for Newton recognition without additional asymptotic hypotheses.

The closest plausible nontrivial route is consequently a resolution or
asymptotic theorem that produces uniform Newton comparison from independently
verifiable analytic structure.  No theorem of that generality is asserted
here.

\subsection{Recognition problem}

\begin{problem}[Newton-comparability recognition]
\label{prob:newton-recognition}
Characterize natural classes of positive singular conformal densities
\(\rho\) for which there exist a compact positive exponent measure \(\mu\)
and constants \(0<c\le C<\infty\) satisfying
\[
cF_\mu\le \rho\le CF_\mu
\]
near a prescribed corner.
\end{problem}

Potential sources include generalized analytic monomialization, generalized
power-series principal parts, and multivariable regularly varying densities
with uniform projective control. The results above suggest that exponent
measures may serve as asymptotic models for such classes, but no recognition
theorem of this generality is asserted here. Establishing one would extend
the present completion classification beyond exact exponent-measure
densities while preserving its intrinsic metric content.

\end{document}